\documentclass[12pt,reqno]{amsart}
\usepackage[margin=0.8in, includeheadfoot]{geometry}
\usepackage{mathtools, amssymb, mathrsfs}
\usepackage[backref=page]{hyperref}
\usepackage[abbrev,lite,msc-links,backrefs]{amsrefs}
\makeatletter
\renewenvironment{proof}[1][\proofname]{%
  \par\pushQED{\qed}\normalfont
  \topsep6\p@\@plus6\p@\relax
  \trivlist
  \item[\hskip\labelsep\textbf{#1}\@addpunct{.}]
}{%
  \popQED\endtrivlist\@endpefalse
}
\makeatother

\usepackage{etoolbox}
\usepackage{xparse}
\newcounter{step}
\pretocmd{\proof}{\setcounter{step}{0}}{}{}
\NewDocumentCommand{\step}{o}{%
  \refstepcounter{step}%
  \par\medskip
  \noindent\textbf{Step \thestep.}%
  \textbf{\IfValueT{#1}{~#1.}}%
  \par\medskip%
}
\makeatletter
\def\amsartsection{\@startsection{section}{1}%
  \z@{2.5\linespacing\@plus\linespacing}{1\linespacing}%
  {\normalfont\scshape\centering}}
\def\section{\vspace*{-34pt}\vspace{\baselineskip}\let\section\amsartsection\section}
\makeatother

\usepackage[nameinlink]{cleveref}
\crefname{step}{Step}{Steps}
\Crefname{step}{Step}{Steps}

\crefname{thm}{Theorem}{Theorems}
\Crefname{thm}{Theorem}{Theorems}

\crefname{prop}{Proposition}{Propositions}
\Crefname{prop}{Proposition}{Propositions}

\crefname{lem}{Lemma}{Lemmas}
\Crefname{lem}{Lemma}{Lemmas}

\crefname{cor}{Corollary}{Corollaries}
\Crefname{cor}{Corollary}{Corollaries}

\crefname{rmk}{Remark}{Remarks}
\Crefname{rmk}{Remark}{Remarks}

\crefname{defn}{Definition}{Definitions}
\Crefname{defn}{Definition}{Definitions}

\crefname{section}{Section}{Sections}
\Crefname{section}{Section}{Sections}

\crefname{subsection}{Section}{Sections}
\Crefname{subsection}{Section}{Sections}

\crefname{subsubsection}{Section}{Sections}
\Crefname{subsubsection}{Section}{Sections}

\crefname{appendix}{Appendix}{Appendices}
\Crefname{appendix}{Appendix}{Appendices}

\hypersetup{
  colorlinks,
  linkcolor=cyan
}
\newcommand{\lh}{\,\rule[-0.05em]{.8em}{.03em}\rule[-.05em]{.03em}{.6em}\,}
\let\div\relax
\DeclareMathOperator{\div}{\operatorname{div}}
\DeclareMathOperator{\curl}{\operatorname{curl}}
\DeclareMathOperator{\tr}{\operatorname{tr}}

\newtheorem{thm}{Theorem}[section]
\newtheorem{lem}[thm]{Lemma}
\newtheorem{cor}[thm]{Corollary}
\newtheorem{prop}[thm]{Proposition}
\newtheorem{defn}[thm]{Definition}
\newtheorem{rmk}[thm]{Remark}
\newtheorem*{clm}{Claim}
\numberwithin{equation}{section}

\BibSpec{misc}{%
  +{} {\PrintAuthors} {author}
  +{,} { \textit} {title}
  +{,} { } {date}
  +{,} { } {eprint}
  +{,} { } {note}
  +{.} {} {transition}
}

\title{The quantitative behavior of \texorpdfstring{\( \alpha \)}{α}-Yang-Mills-Higgs fields on surfaces}   
\date{\today}
\keywords{\texorpdfstring{\( \alpha \)}{α}-Yang-Mills-Higgs fields; Energy-identity; no-neck property; hidden Jocobian structure; Pohozaev equality}
\subjclass[2020]{58E15\textbullet 53C43}

\author{Wanjun Ai}
\email{wanjunai@swu.edu.cn}
\address{School of Mathematics and Statistics\\
  Southwest University\\
  Chongqing\\
  400715\\
  P. R. China
}
\author{Jiayu Li}
\email{jiayuli@ustc.edu.cn}
\address{School of Mathematical Sciences\\
  University of Science and Technology of China\\
  Hefei 230026\\
  P. R. China
}
\author{Miaomiao Zhu}
\email{mizhu@sjtu.edu.cn}
\address{
  School of Mathematical Sciences\\
  Shanghai Jiao Tong University\\
  800 Dongchuan Road\\
  Shanghai 200240\\
  P. R. China
}

\begin{document}
\begin{abstract}
  We investigate the blow-up behavior of \( \alpha \)-Yang--Mills--Higgs (\( \alpha \)-YMH) fields over closed Riemannian surfaces with the target fiber \( F = S^{K-1} \subset \mathbb{R}^K \) being the round sphere, focusing on the establishment of the \( \alpha \)-energy identity and the no-neck property during the bubbling process. A central innovation is the identification of a hidden Jacobian structure through Hodge decomposition and a new conservation law. Furthermore, we derive a Pohozaev-type identity for \( \alpha \)-YMH fields, which enables refined control of the energy density. Together, these advances ensure the validity of the \( \alpha \)-energy identity as \( \alpha \to 1 \). Our analysis further sharpens the Lorentz space estimates from the \( L^{2,\infty} \) to the optimal \( L^{2,1} \) scale, ultimately yielding the no-neck property in the blow-up regime. These results provide a unified and quantitative framework for understanding singularity formation in variational gauge theories.
\end{abstract}
\maketitle
\section{Introduction}\label{sec:intro}
Let \( \Sigma \) be a Riemannian manifold with metric \( g \), \( G \) be a compact connected Lie group with Lie algebra \( \mathfrak{g} \), and \( \mathcal{P} \) be a \( G \)-principal bundle over \( \Sigma \). Let \( F \) be a Riemannian manifold admitting a \( G \)-action, and let \( \mathcal{F} = \mathcal{P} \times_G F \) be the associated fiber bundle. A generalized Higgs potential \( \mu \) is a smooth gauge-invariant function on \( \mathcal{F} \). Let \( \mathscr{S} \) and \( \mathscr{A} \) denote the space of smooth sections of \( \mathcal{F} \) and the affine space of smooth connections on \( \mathcal{P} \), respectively. The Yang--Mills--Higgs (YMH) functional is defined by
\[
  \mathcal{L}(A,\phi):=\frac{1}{2}\int_\Sigma\left(  \lvert \nabla _A\phi \rvert^2 +\lvert F_A \rvert^2+\lvert \mu(\phi) \rvert^2 \right)dv_g,\quad (A,\phi)\in \mathscr{A}\times \mathscr{S},
\]
where \( \nabla_A \) is the covariant differential induced by \( A \), and \( F_A \) is the corresponding curvature 2-form, i.e., \( F_A \in \Omega^2(\mathfrak{g}_{\mathcal{P}}) \), with \( \mathfrak{g}_{\mathcal{P}}\mathpunct{:}= \mathcal{P} \times_{\mathrm{Ad}} \mathfrak{g} \). Critical points of \( \mathcal{L} \) are referred to as \emph{YMH fields}.

 YMH theory has played a foundational role not only in particle physics and quantum field theory \cites{JaffeTaubes1980Vortices,BethuelBrezisHelein1994GinzburgLandau,deGennes1999superconductivity}, but also in geometric analysis and differential geometry. Critical points of the YMH functional encompass a broad range of important geometric objects, including Ginzburg--Landau vortices, (pseudo-)holomorphic curves, (gauged) harmonic maps, and pure Yang--Mills fields \cites{BethuelBrezisHelein1994GinzburgLandau,CieliebakGaioRieraSalamon2002Symplectic,Riera2003Hamiltonian,LinYang2003Gauged,TianYang2002Compactification}. In particular, in two dimensions, the conformal invariance of the Higgs energy term \( \lVert \nabla_A \phi \rVert_{L^2} \) makes the compactness and blow-up behavior of YMH fields central to understanding the structure of their moduli spaces \cites{RieraTian2009Compactification,Song2016Convergence}. Such compactness results are essential, for instance, in defining Hamiltonian Gromov--Witten (HGW) invariants \cite{Riera2003Hamiltonian}.

On the other hand, the existence theory for general critical points of the YMH functional over Riemannian surfaces can be approached via a min--max framework applied to the perturbed \( \alpha \)-YMH functional, inspired by the seminal work of Sacks and Uhlenbeck \cite{SacksUhlenbeck1981existence} on \( \alpha \)-harmonic maps. More precisely, consider the \( \alpha \)-Yang--Mills--Higgs (\( \alpha \)-YMH) functional defined as
\[
  \mathcal{L}_\alpha(A,\phi):=\frac{1}{2}\int_\Sigma\left( \left( 1+\lvert \nabla _A\phi \rvert^2 \right)^\alpha - 1 + \lvert F_A \rvert^2 + \lvert \mu(\phi) \rvert^2 \right)dv_g,\quad (A,\phi)\in \mathscr{A} \times \mathscr{S}.
\]
Critical points of \( \mathcal{L}_\alpha \) are called \emph{\( \alpha \)-YMH fields} and satisfy the following Euler--Lagrange equations on \( \Sigma \):
\begin{equation}\label{eq:alpha-YMH}
  \begin{cases}
    \nabla _A^*\left( \alpha(1+\lvert \nabla _A\phi \rvert^2)^{\alpha-1}\nabla _A\phi \right)+\mu(\phi)\cdot \nabla \mu(\phi)=0,\\
    D_A^*F_A+\alpha(1+\lvert \nabla _A\phi \rvert^2)^{\alpha-1}\phi^*\nabla _A\phi=0,
  \end{cases}
\end{equation}
where \( D_A \) is the covariant exterior differential and \( D_A^* \) is its adjoint. The term \( \phi^*\nabla_A\phi \) takes values in the dual space of \( \Omega^1(\mathfrak{g}_{\mathcal{P}}) \); that is, for any \( b \in \Omega^1(\mathfrak{g}_{\mathcal{P}}) \),
\[
  \langle \phi^*\nabla _A\phi,b \rangle=\langle \nabla _A\phi,b\phi \rangle.
\]
We refer to\cites{Song2011Critical,AiSongZhu2019boundary} for more details.

A key ingredient in the theory is again the compactness of critical points of the \( \alpha \)-YMH functional. Clearly, when \( \alpha=1 \), the \( \alpha \)-YMH fields reduce to YMH fields. Moreover, locally, when \( A \) is trivial and the Higgs potential \( \mu = 0 \), the \( \alpha \)-YMH functional reduces to the classical Dirichlet energy, and \eqref{eq:alpha-YMH} coincides with the equation for \( \alpha \)-harmonic maps.

Under various conditions, there is no energy loss during the bubbling convergence of \( \alpha \)-harmonic maps with uniformly bounded energy to their weak limits--i.e., the energy identity holds when blow-up occurs; see \citelist{\cite{Jost1991Two}\cite{ChenTian1999Compactification}\cite{Moore2007Energy}\cite{Moore2017Introduction}\cite{LiWang2010Weak}\cite{Lamm2010Energy}\cite{LiLiuWang2017Blowup}\cite{LiZhu2019Energy}\cite{JostLiuZhu2022Asymptotic}\cite{LiuZhu2024Asymptotic}}. In particular, when the target manifold is a round sphere, the high degree of symmetry of the target ensures that both the energy identity and the no-neck property hold \cite{LiZhu2019Energy}. 

Motivated by these developments, the compactness of \( \alpha \)-YMH fields on Riemannian surfaces has been studied. By imposing some entropy-type condition analogous to the \( \alpha \)-harmonic map setting, the energy identity of \( \alpha \)-YMH fields on closed Riemannian surfaces was established in \cite{Song2011Critical}. The free boundary value problem for \( \alpha \)-YMH fields on Riemannian surfaces has also been considered, where bubbling convergence is achieved, possibly involving the splitting off of harmonic discs with free boundary near points of energy concentration on the boundary \cite{AiSongZhu2019boundary}.

In this paper, we focus on the blow-up behavior of \( \alpha \)-Yang--Mills--Higgs fields on closed Riemannian surfaces, aiming to establish the energy identity and no-neck property under different assumptions. We consider a general sequence of \( \alpha \)-YMH fields \( \{ (A_\alpha, \phi_\alpha) \} \) with uniformly bounded \( \alpha \)-energy. As shown in \cite{Song2011Critical}, there exists a subsequence \( \alpha_n \to 1 \) and a sequence of gauge transformations \( S_n \) such that \( S_n^* A_{\alpha_n} \) converges to a limiting connection \( A_\infty \), and \( S_n^* \phi_{\alpha_n} \) converges to a limiting section \( \phi_\infty \), away from a finite blow-up set \( \mathcal{S} \subset \Sigma \). Moreover, \( (A_\infty, \phi_\infty) \) extends to a smooth YMH field on all of \( \Sigma \), and at each blow-up point \( x \in \mathcal{S} \), a nontrivial harmonic sphere may appear. It is standard in this setting to establish the following foundational result:
\begin{prop}[\cites{Song2011Critical}]\label{prop:bubbing}
  Suppose \( (A_\alpha,\phi_\alpha)\in \mathscr{A}\times \mathscr{S} \) is a sequence of smooth \( \alpha \)-YMH fields with \( \mathcal{L}_\alpha(A_\alpha,\phi_\alpha)\leq\Lambda<+\infty \). Then,
  \begin{enumerate}
    \item  \label{item:prop-small-energy-I} the set
      \[
	\mathcal{S}:=\left\{ x\in\Sigma:\lim_{r\to0}\liminf_{\alpha\to1}\int_{D_r(x)}\lvert \nabla_{A_\alpha}\phi_\alpha \rvert^2\geq\epsilon_0 \right\}
      \]
      is a set of at most finitely many points;
    \item \label{item:prop-small-energy-II} as \( \alpha\to1 \), there exists a subsequence of \( \left\{ (A_\alpha,\phi_\alpha) \right\} \) (which we will not distinguish) and a sequence of smooth gauge transformations \( \left\{ S_\alpha \right\} \), such that
      \[
	S_\alpha^*A_\alpha\to A_\infty\in C_{\mathrm{loc}}^\infty(\Sigma\setminus \mathcal{S})\cap C^0(\Sigma),\quad
	S_\alpha^*\phi_\alpha\to \phi_\infty\in C_{\mathrm{loc}}^\infty(\Sigma\setminus \mathcal{S}),
      \]
      and \( (A_\infty,u_\infty) \) extends to a smooth YMH fields on \( \Sigma \);
    \item \label{item:prop-small-energy-III}for each \( x^j\in \mathcal{S} \), there exists a sequence of points \( \left\{ x_\alpha^j \right\}\subset\Sigma \), a sequence of positive numbers \( \left\{ r_\alpha^j \right\} \) and a non-trivial harmonic spheres \( \omega^j:S^2\to F \) such that
      \begin{enumerate}
	\item \( x_\alpha^j\to x^j \) and \( r_\alpha^j\to 0 \) as \( \alpha\to1 \);
	\item for any \( i\neq j \), we have
	  \[
	    \lim_{\alpha\to1}\left( \frac{r_\alpha^i}{r_\alpha^j}+\frac{r_\alpha^j}{r_\alpha^i}+\frac{\lvert x_\alpha^i-x_\alpha^j \rvert}{r_\alpha^i+r_\alpha^j} \right)=+\infty;
	  \]
	\item there are at most finite many harmonic spheres \( \left\{ \omega^j \right\} \).
      \end{enumerate}
  \end{enumerate}
\end{prop}

The main goal of this work is to establish that, when the domain is a closed Riemannian surface and the target fiber \( F = S^{K-1} \subset \mathbb{R}^K \) is the round sphere, both the energy identity and the no-neck property hold for sequences of \( \alpha \)-YMH fields with uniformly bounded energy. This extends key techniques from \( \alpha \)-harmonic map theory to the gauge-invariant, curvature-coupled setting of \( \alpha \)-YMH theory.

To formulate our result, we define for any open set \( \Omega \subset \mathbb{R}^2 \) the classical and gauged energy functionals by
\[
  E(u, \Omega)    \mathpunct{:}= \frac{1}{2} \int_\Omega \lvert du \rvert^2 \, dv_g, \qquad
  E(A, u; \Omega) \mathpunct{:}= \frac{1}{2} \int_\Omega \lvert \nabla_A u \rvert^2 \, dv_g,
\]
where \( u : \Omega \to S^{K-1} \) and \( A \) is a connection on the trivial bundle over \( \Omega \).
\begin{thm}\label{thm:main}
  Suppose \( (A_\alpha, u_\alpha) \in \mathscr{A} \times \mathscr{S} \) is a sequence of smooth local \( \alpha \)-YMH fields over the unit disc \( D \), satisfying \( \mathcal{L}_\alpha(A_\alpha, u_\alpha) \leq \Lambda < +\infty \). Furthermore, assume that the fiber \( F \) of the associated bundle \( \mathcal{F} \) is \( S^{K-1}\subset \mathbb{R}^K \). Then the following properties hold:
  \begin{enumerate}
    \item \emph{energy identity}:
      \begin{equation}\label{eq:alpha-EI}
	\lim_{\alpha\to1}\mathcal{L}_\alpha(A_\alpha,u_\alpha;D)=\mathcal{L}(A_\infty,u_\infty;D)+\sum_{j}E(\omega^j;\mathbb{R}^2).
      \end{equation}
    \item \emph{no neck property}: for all \( j \),
      \begin{equation}\label{eq:NN}
	\lim_{\delta\to0}\lim_{R\to\infty}\lim_{\alpha\to1}\mathrm{osc}_{D_\delta(x^j)\setminus D_{r_\alpha^j R}(x_\alpha^j)}u_\alpha=0.
      \end{equation}
  \end{enumerate}
\end{thm}

For general target manifolds without symmetry, there are counter examples that the energy identity and no-neck property fails for \( \alpha \)-harmonic maps \cites{LiWang2010Weak,LiWang2015counterexample}. The advantage of spherical target case is that a hidden Jacobian structure for the Euler--Lagrange equations can be recovered. This particular structure has been explored in various harmonic maps type problems, see e.g., \cites{Bethuel1992Un,Helein2002Harmonic,LinRiviere2002Energy,LiZhu2019Energy,LiLiuZhuZhu2021Energy,LiZhuZhu2023Qualitative}. 

For \( \alpha \)-harmonic maps into sphere, the hidden Jacobian structure is derived as follows \cites{Helein2002Harmonic,LiZhu2019Energy}: Recall the Euler--Lagrange equation for such kind of harmonic maps is given by 
\[
  \div(f_\alpha \nabla u)+f_\alpha\lvert \nabla u \rvert^2u=0,\quad f_\alpha \mathpunct{:}=\alpha(1+\lvert \nabla u \rvert^2)^{\alpha-1}.
\]
Since \( u\in S^{K-1} \), we have \( \lvert u \rvert^2=1 \) and \( \sum_{j}u^j\nabla  u^j=0 \), thus we can rewrite the equation of spherical \( \alpha \)-harmonic maps as 
\begin{equation}\label{eq:skew-div-form-alpha-HM}
  \div\left( f_\alpha \nabla u^i \right)=-f_\alpha \nabla u^j\nabla u^ju^i 
  =f_\alpha \left( \nabla u^iu^j-\nabla u^ju^i \right)\nabla u^j
  =\mathpunct{:}\Omega^{ij}\nabla u^j,
\end{equation}
where \( \Omega^{ij}=f_\alpha(\nabla u^iu^j-\nabla u^ju^i) \) is skew-symmetric and divergence free, i.e., \( \div\Omega^{ij}=0 \). 
Thus, we may find \( G^{ij} \in L^{\frac{2\alpha}{2\alpha-1}}(\mathbb{R}^2) \) such that \( \Omega^{ij} = \nabla^\perp G^{ij} \) in \( D \). This allows us to recover the hidden Jacobian structure
\[
  \div(f_\alpha \nabla u^i) = \Omega^{ij} \nabla u^j = \nabla^\perp G^{ij}\cdot \nabla u^j.
\]
In particular, this implies that \( \nabla^\perp G^{ij} \cdot \nabla u^j \) belongs to the Hardy space and the compensate compactness results can be applied, which is a strictly finer space than \( L^1 \)—the natural target space for \( u \in W^{1,2\alpha} \)—since \( \nabla^\perp G^{ij} \in L^{\frac{2\alpha}{2\alpha-1}} \) and \( \nabla u^j \in L^{2\alpha} \).

Similarly, the Euler--Lagrange equation for spherical \( \alpha \)-Yang--Mills--Higgs fields can also be expressed in a form analogous to \eqref{eq:skew-div-form-alpha-HM}. In fact, starting from the localized expression of \eqref{eq:alpha-YMH} in the case of spherical fibers (see \eqref{eq:alpha-YMH-local}), we obtain
\begin{equation*}
\operatorname{div}(f_\alpha \nabla_A u) + f_\alpha \lvert \nabla_A u \rvert^2 u + f_\alpha a \cdot \nabla_A u - \mu(u) \nabla \mu(u) = 0,\quad f_\alpha \mathpunct{:}=\alpha(1+\lvert \nabla _Au \rvert^2)^{\alpha-1}.
\end{equation*}
Observe that for a metric connection \( A \), the matrix \( a \) is skew-symmetric, and we have \( \sum_j \nabla_A u^j u^j = \sum_j \nabla u^j u^j = 0 \). This yields
\begin{equation*}
\operatorname{div}(f_\alpha \nabla_A u^i)
= \Omega^{ij} \nabla_A u^j - \mu(u)\nabla_i \mu(u),
\end{equation*}
where
\begin{equation*}
\Omega^{ij} = f_\alpha \left( \nabla_A u^i u^j - \nabla_A u^j u^i - a^{ij} \right)
\end{equation*}
is still skew-symmetric, though generally no longer divergence-free. This lack of divergence-free structure introduces new analytic challenges and makes the corresponding theory particularly rich and interesting.

To recover the compensation-type compactness available for spherical \( \alpha \)-harmonic maps, we establish a new conservation law for \( \alpha \)-Yang--Mills--Higgs fields. This law arises, inspired by Noether's theorem, from the invariance of the action under the global gauge transformation \( g \cdot (a, u) = (gdg^{-1} +gag^{-1}, gu) \) for \( g \in G \subset \mathrm{SO}(K) \). More precisely,
\begin{thm}[Conservation Law for \( \alpha \)-YMH Fields with fiber \( F=S^{K-1} \)]\label{thm:conservation-law}
Let \( u: B^m \to S^{K-1}\subset\mathbb{R}^K \) be a local section on unit ball \( B^m\subset \mathbb{R}^m \) and \( A = d+a_\alpha dx^\alpha \in \Omega^1(B^m, \mathfrak{so}(K)) \) be a connection, if \( (A,u) \) is a critical point of the \( \alpha \)-YMH functional 
\[
  \mathcal{L}_\alpha(A,u;B^m)=\frac{1}{2}\int_{B^m}\left(  1+\lvert \nabla _Au \rvert^2  \right)^\alpha-1+\lvert F_A \rvert^2+\lvert \mu(u) \rvert^2,
\]
then the following \emph{conservation law} holds:
\begin{equation}\label{eq:conservation-law-general-b}
\mathrm{div}\left[\left(f_\alpha [\nabla_A u]^p_\beta\, b_q^p u^q + F_{\gamma\beta,r}^q(a_{\gamma,p}^r b_q^p - b_p^r a_{\gamma,q}^p)\right)\partial_{x^\beta} \right] = 0,\quad \forall b = (b_q^p) \in \mathfrak{so}(K),
\end{equation}
where the curvature components \( F_{\alpha\beta} \in \mathfrak{so}(K) \) are given by
\[
F_{\alpha\beta} = \partial_\alpha a_\beta - \partial_\beta a_\alpha + [a_\alpha, a_\beta],
\]
and we set 
\[
[\nabla_A u]^i_\beta = \partial_\beta u^i + a_{\beta,j}^i u^j,\quad 
f_\alpha := \alpha\left(1 + |\nabla_A u|^2\right)^{\alpha - 1}.
\]
\end{thm}
In applications to our setting of \( \alpha \)-YMH Fields with fiber \( F=S^{K-1} \), we need the following specific component form of conservation law.
\begin{cor}\label{cor:conservation-law-component}
In particular, taking \( b = E_{ij} - E_{ji} \in \mathfrak{so}(K) \) (the elementary skew-symmetric matrix with 1 in the \( (i,j) \)-position and \( -1 \) in the \( (j,i) \)-position), we obtain the component-wise identity:
\[
\mathrm{div}\left[ \left( f_\alpha \left( [\nabla_A u]^i_\beta u^j - [\nabla_A u]^j_\beta u^i \right) - 2[F_{\gamma\beta}, a_\gamma]_j^i \right)\partial_{x^\beta} \right] = 0,
\] 
or more concisely
\begin{equation}\label{eq:conservation-law-component}
  \div\left[ f_\alpha\left( [\nabla _Au]^iu^j-[\nabla _Au]^ju^i \right)+2 \mathrm{tr}[F_A,a]^i_j \right]=0,
\end{equation}
where \( \mathrm{tr}[F_A,a]=[F_{\beta\gamma},a_\gamma]dx^\beta\in\Omega^1(B^m, \mathfrak{so}(K)) \) is a globally defined Lie-algebra-valued 1-form.
\end{cor}
\begin{rmk}
  When the curvature vanishes, i.e., \( F_A = 0 \), the connection \( A \) is locally gauge equivalent to the trivial connection. That is, there exists a local gauge in which the connection 1-form \( a \) vanishes. In such a gauge, the covariant derivative reduces to the ordinary derivative, \( \nabla_A u = du \), and the conservation law simplifies to the classical identity for \( \alpha \)-harmonic maps into spheres:
  
  \begin{equation}\label{eq:conservation-law-alpha-HM}
    \div\left[ f_\alpha(\nabla u^iu^j-\nabla u^ju^i) \right]=0.
  \end{equation}
\end{rmk}

By applying the conservation law, the equation for \( \alpha \)-YMH fields can be recast into a form that exhibits a Jacobian structure, accompanied by additional lower-order terms arising from the connection 1-form \( a \). A crucial advantage lies in the favorable regularity theory for \( a \): due to the availability of local Coulomb gauges and the fact that \( a \) satisfies a perturbed Yang--Mills equation, one obtains small energy regularity for \( a \) over the entire disc \( D \). In contrast, such small energy regularity does not hold globally for the map \( u \), which only admits small energy control on annular subregions \( T_{r_1, r_2} \subset D \) over the neck, provided the ratio \( r_1/r_2 \) is fixed. This observation enables us to view the \( \alpha \)-YMH equation as a system governed by an underlying Jacobian structure, perturbed by well-controlled terms involving \( a \). Consequently, we can recover the compensation-compactness property and establish the energy identity.

To establish the energy identity, we first apply small energy estimates to obtain a uniform \( C^0 \) bound on \( f_\alpha \) throughout \( D \). However, to prove both the energy identity for the \( \alpha \)-energy and the absence of necks, it is necessary to strengthen this estimate and show that the bound converges to \( 1 \). In the case of \( \alpha \)-harmonic maps, this refinement relies on a Pohozaev-type identity. For \( \alpha \)-YMH fields, we derive a Pohozaev identity for the coupled system involving both maps and connections.
\begin{prop}[Pohozaev identity for \( \alpha \)-YMH fields]\label{prop:Pohozaev}
Let \( D\subset \mathbb{R}^2 \) be the unit disk equipped with the standard Euclidean metric. Suppose \( (A_\alpha,u_\alpha) \) is a critical point of \( \mathcal{L}_{\alpha}(A,u;D) \). Then, for any \( 0<t<1 \), the following Pohozaev identity holds:
\begin{equation}\label{eq:Pohozaev}
  \begin{split}
      &\int_{\partial D_t}\left( 1+\lvert \nabla_{A_\alpha} u_\alpha \rvert^2 \right)^{\alpha-1}
      \left( \lvert \nabla _{A_\alpha;\partial_r}u_\alpha \rvert^2
      -\frac{1}{2\alpha}\lvert \nabla_{A_\alpha} u_\alpha \rvert^2 \right)d\sigma_0\\
      &\qquad=\frac{\alpha-1}{\alpha t}\int_{D_t}\left( 1+\lvert \nabla_{A_\alpha} u_\alpha \rvert^2 \right)^{\alpha-1}\lvert \nabla_{A_\alpha} u_\alpha \rvert^2 dv_g+O(t),
  \end{split}
\end{equation}
where \( d\sigma_0 \) denotes the volume form on \( \partial D_t \).
\end{prop}
We remark that the derivation of the Pohozaev identity is based on variations of the domain, in contrast to conservation laws, which stem from symmetries of the target manifold. For the case of \( \alpha \)-harmonic maps, we refer the reader to \cite{LiWang2010Weak}, and note that a suitable modification of this argument also yields the monotonicity formula in higher dimensions; see \cite{Price1983Monotonicity} for Yang--Mills fields and \cite{Zhang2004Compactness} for Yang--Mills--Higgs fields.

The remainder of this paper is organized as follows. In \cref{sec:Pre}, we review fundamental properties of \( \alpha \)-YMH fields and key elliptic estimates in Lorentz spaces. In \cref{sec:EL-ConservationLaw-Pohozaev}, we then derive the Euler--Lagrange equation for \( \alpha \)-YMH fields in the case where the fiber \( F \) is a sphere. This is followed by two crucial lemmas: a conservation law for \( \alpha \)-YMH fields, and a Pohozaev-type identity, which are central to our arguments for the energy identity and no-neck property, respectively. In \cref{sec:EI}, we establish the energy identity, while in \cref{sec:NN}, we prove the no-neck property. Together, these results complete the proof of \cref{thm:main}.
\section{Preliminaries}\label{sec:Pre}
In this section, we begin by recalling the small energy estimates for the coupled \( \alpha \)-Yang--Mills--Higgs system. In the first subsection, we extend these estimates to a rescaled version that is both suitable and essential for the blow-up analysis. As an application, we conclude the subsection with a sketch of the proof of \cref{prop:bubbing}. In the second subsection, we revisit several classical results concerning elliptic estimates in Lorentz spaces, which will play a key role in our later arguments.

\subsection{Small energy estimates}
We start by considering the local formulation of the \( \alpha \)-YMH equations with target fiber \( F = S^{K-1} \) in a neighborhood of a blow-up point. Assume \( \Sigma = D \subset \mathbb{R}^2 \) is the unit disc equipped with the Euclidean metric, and that both the principal bundle \( \mathcal{P} \) and the associated fiber bundle \( \mathcal{F} = \mathcal{P} \times_G F \) are trivial. Under a fixed trivialization, a section \( \phi \in \mathscr{S} \) can be identified with a map \( u : D \to F \), and a connection \( A \in \mathscr{A} \) is written as \( A = d + a \) for a \( \mathfrak{g} \)-valued 1-form \( a \). After imposing the Coulomb gauge \cite{Uhlenbeck1982Connections}, the Euler--Lagrange equations \eqref{eq:alpha-YMH} reduce locally to
\begin{equation}\label{eq:alpha-YMH-local}
  \begin{cases}
    \div(f_\alpha \nabla_A u) + f_\alpha \lvert \nabla_A u \rvert^2 u + \Phi_\alpha(A,u) = 0, \\
    \Delta a + \Psi_\alpha(A,u) = 0,
  \end{cases}
\end{equation}
where \( f_\alpha = \alpha (1 + \lvert \nabla_A u \rvert^2)^{\alpha - 1} \), and
\begin{align*}
  \Phi_\alpha(A,u) &= f_\alpha \left( a \cdot du + a \cdot a \cdot u \right) - \mu(u) \cdot \nabla \mu(u), \\
  \Psi_\alpha(A,u) &= -f_\alpha u^* \nabla_A u + \langle da, a \rangle + \langle a, [a, a] \rangle - [F_A \lh a^\sharp].
\end{align*}
We refer the reader to \cref{sec:EL-ConservationLaw-Pohozaev} for the derivation of \eqref{eq:alpha-YMH-local}, as well as the precise definitions of \( a \cdot du \) and \( a \cdot a \cdot u \).

The following small energy estimates are analogues of those established for harmonic maps by Sacks--Uhlenbeck \cite{SacksUhlenbeck1981existence}; we refer to \cite{Song2011Critical} for the corresponding results in the context of \( \alpha \)-YMH fields.

\begin{prop}[Small energy estimates; \cite{Song2011Critical}]\label{prop:small-energy-estimate}
Let \( D \subset \mathbb{R}^2 \) be the unit disc and \( D_{1/2} \) the concentric disc of radius \( 1/2 \). Suppose that \( (A,u) \) is a smooth pair satisfying \eqref{eq:alpha-YMH-local} in \( D \) and \( \lVert F_A \rVert_{L^2(D)} \leq \delta_0 \), where \( \delta_0 > 0 \) is the constant from the Coulomb gauge theorem.
\begin{enumerate}
  \item There exist constants \( \alpha_0 > 1 \) and \( \epsilon_0 > 0 \) such that, if \( \lVert \nabla_A u \rVert_{L^2(D)} < \epsilon_0 \), then for any \( 1 < p < \infty \), the following estimate holds uniformly for \( 1 \leq \alpha < \alpha_0 \):
  \[
    \lVert u - \bar{u} \rVert_{W^{2,p}(D_{1/2})} \leq C \left( \lVert \nabla_A u \rVert_{L^2(D)} + \lVert F_A \rVert_{L^2(D)} \right),
  \]
  where \( \bar{u} \) denotes the integral mean value of \( u \) over \( D \), and \( C > 0 \) depends on \( p \), \( \alpha_0 \), \( \epsilon_0 \), and \( \lVert \mu \rVert_{W^{1,\infty}} \).

  \item For any \( 1 < p < 2 \), there exists \( \alpha_p \in (1, 2) \) such that, for all \( 1 < \alpha < \alpha_p \), the following estimate holds:
  \[
    \lVert a \rVert_{W^{2,p}(D_{1/2})} \leq C \left( \lVert \nabla_A u \rVert_{L^2(D)} + \lVert F_A \rVert_{L^2(D)} \right),
  \]
  where \( C > 0 \) depends on \( p \), \( \alpha_0 \) and \( \lVert \mu \rVert_{W^{1,\infty}} \).
\end{enumerate}
\end{prop}
\begin{rmk}
  By a bootstrap argument, the above theorem implies the following higher-order estimate
  \[
    \lVert u-\bar u \rVert_{W^{k,p}(D_{1/2})}+\lVert a \rVert_{W^{k,p}(D_{1/2})}\leq C\left( \lVert \nabla_Au \rVert_{L^2(D)}+\lVert F_A \rVert_{L^2(D)} \right),
  \]
  and 
  \[
    \sup_{D_{1/2}}\lvert \nabla _A^k(u-\bar u) \rvert+\sup_{D_{1/2}}\lvert \nabla _A^kF_A \rvert \leq C\left( \lVert \nabla_Au \rVert_{L^2(D)}+\lVert F_A \rVert_{L^2(D)} \right),
  \]
  where \( p>1 \) and \( k=2,3,\ldots \), \( C \) is a constant depending on \( k,p,\alpha_0,\epsilon_0, \lVert \mu \rVert_{W^{1,\infty}} \).
\end{rmk}
In the blow-up analysis, unlike the cases of harmonic maps or Dirac-harmonic maps, the \( \alpha \)-YMH functional is not scaling invariant. To address this, we introduce the following notions for the rescaled fields.
\begin{defn}
  Let \( (A,u) \) be a local section of \( \mathscr{A} \times \mathscr{S} \) over $D$ and \( \epsilon_\alpha\in(0,1] \), \( \lambda>0 \). The \emph{generalized \( \alpha \)-YMH functional} is defined by
  \[
    \mathcal{L}_{\alpha,\epsilon_\alpha}(A,u;D) := \frac{1}{2} \int_{D} \left( \left( \epsilon_\alpha + |\nabla_A u|^2 \right)^\alpha + \epsilon_\alpha^{\alpha-2} |F_A|^2 + \epsilon_\alpha^\alpha |\mu(u)|^2 \right) dx.
  \]
  A critical point \( (A,u) \) of \( \mathcal{L}_{\alpha,\epsilon_\alpha} \) is called a \emph{generalized \( \alpha \)-YMH field}.

  Moreover, by replacing \( \epsilon_\alpha \) with \( \lambda^2 \epsilon_\alpha \) in the definition above, we obtain the \emph{rescaled generalized \( \alpha \)-YMH functional} \( \mathcal{L}_{\alpha, \lambda^2 \epsilon_\alpha}, \) whose critical points are called \emph{rescaled generalized \( \alpha \)-YMH fields}, where \( \lambda > 0 \) is referred to as the \emph{rescaling factor}.
\end{defn}
Clearly, when \( \epsilon_\alpha = 1 \), the functional \( \mathcal{L}_{\alpha,\epsilon_\alpha} \) reduces to the standard \( \alpha \)-YMH functional over \( D \) (up to a constant). Moreover, if \( (A,u) \) is a generalized \( \alpha \)-YMH field, then under a conformal transformation of the domain \( D \), the pair \( (A,u) \) is transformed into a rescaled generalized \( \alpha \)-YMH field.

Specifically, consider the scaling map \( \rho_\lambda : x \mapsto \lambda x \), and define the rescaled fields by
\[
  a_\lambda(x) := \rho_\lambda^* a(x) = a_i(\lambda x)\lambda dx^i = \lambda a(\lambda x),\quad 
  u_\lambda(x) := \rho_\lambda^* u(x) = u(\lambda x).
\]
Then it is straightforward to verify that
\begin{align*}
  \nabla_{A_\lambda} u_\lambda(x) &= du_\lambda(x) + a_\lambda(x) \cdot u_\lambda(x) = \lambda \nabla_A u(\lambda x),\\
  F_{A_\lambda}(x) &= d a_\lambda(x) + a_\lambda(x) \wedge a_\lambda(x) = \lambda^2 F_A(\lambda x).
\end{align*}
Assuming \( (A,u) \) is a generalized \( \alpha \)-YMH field, we compute
\begin{equation}\label{eq:scaled-energy-relation}
  \begin{split}
      \mathcal{L}_{\alpha,\epsilon_\alpha}(A,u)
    &=\frac{1}{2}\int_D\left( \left( \epsilon_\alpha+\lvert \nabla _Au \rvert^2 \right)^\alpha+\epsilon_\alpha^{\alpha-2}\lvert F_A \rvert^2+\epsilon_\alpha^\alpha\lvert \mu(u) \rvert^2 \right) dy\\
    &=\frac{1}{2}\lambda^{2-2\alpha}\int_{D_{\frac{1}{\lambda}}}\left( \left( \lambda^2\epsilon_\alpha+\lvert \nabla _{A_\lambda}u_\lambda \rvert^2 \right)^\alpha+(\lambda^2\epsilon_\alpha)^{\alpha-2}\lvert F_{A_\lambda} \rvert^2+(\lambda^2\epsilon_\alpha)^\alpha\lvert \mu(u_\lambda) \rvert^2 \right) dx,
  \end{split}
\end{equation}
where we have performed the change of variables \( y = \lambda x \). It follows that \( (A_\lambda, u_\lambda) \) is a rescaled generalized \( \alpha \)-YMH field, corresponding to the rescaled parameter \( \lambda^2\epsilon_\alpha \) and with rescale factor \( \lambda \).

The Euler--Lagrange equation for generalized \( \alpha \)-YMH fields can be deduced from \eqref{eq:alpha-YMH-local}. In particular, observe that if \( (A,u) \) is a critical point of \( \mathcal{L}_{\alpha,1} \), then for any \( \lambda > 0 \), the rescaled pair \( (A_\lambda, u_\lambda) \) is a critical point of \( \mathcal{L}_{\alpha, \lambda^2} \). Thus, the Euler--Lagrange equations for \( \mathcal{L}_{\alpha, \epsilon_\alpha} \) can be deduced by substituting \( \epsilon_\alpha \) for \( \lambda^2 \) in the rescaled system.

By direct computation, starting from \eqref{eq:alpha-YMH-local}, we obtain for the rescaled pair \( (A_\lambda, u_\lambda) \):
\[
  \begin{cases}
    \div\left( f_{\alpha,\lambda}\nabla_{A_\lambda}u_\lambda \right) + f_{\alpha,\lambda} |\nabla_{A_\lambda} u_\lambda|^2 u_\lambda + f_{\alpha,\lambda}\left( a_\lambda \cdot d u_\lambda + a_\lambda \cdot a_\lambda \cdot u_\lambda \right) - \lambda^{2\alpha} \mu(u_\lambda) \cdot \nabla \mu(u_\lambda) = 0, \\[1.5ex]
    \Delta a_\lambda - \lambda^{2-2\alpha} f_{\alpha,\lambda} u_\lambda^* \nabla_{A_\lambda} u_\lambda + \langle d a_\lambda, a_\lambda \rangle + \langle a_\lambda, [a_\lambda, a_\lambda] \rangle - \left[ F_{A_\lambda} \lh a_\lambda^\sharp \right] = 0,
  \end{cases}
\]
where \( f_{\alpha,\lambda} = \alpha \left( \lambda^2 + |\nabla_{A_\lambda} u_\lambda|^2 \right)^{\alpha-1} \). In deriving the second equation, we use the following rescaling identity: for any test matrix-valued 1-form \( b \),
\[
  \langle u^* \nabla_A u, b \rangle := \langle \nabla_A u, b u \rangle = \langle \lambda^{-2} \nabla_{A_\lambda} u_\lambda, \lambda^{-1} b_\lambda u_\lambda \rangle = \lambda^{-3} \langle \nabla_{A_\lambda} u_\lambda, b_\lambda u_\lambda \rangle = \lambda^{-3} \langle u_\lambda^* \nabla_{A_\lambda} u_\lambda, b_\lambda \rangle,
\]
which implies that
\[
  u^* \nabla_A u = \lambda^{-3} u_\lambda^* \nabla_{A_\lambda} u_\lambda.
\]

Therefore, after replacing \( \lambda^2 \) by \( \epsilon_\alpha \) throughout, the Euler--Lagrange equations for the generalized \( \alpha \)-YMH fields become
\begin{equation}\label{eq:generalized-alpha-YMH-local}
  \begin{cases}
    \div \left( f_{\alpha, \epsilon_\alpha} \nabla_A u \right)
    + f_{\alpha, \epsilon_\alpha} |\nabla_A u|^2 u
    + f_{\alpha, \epsilon_\alpha} \left( a \cdot d u + a \cdot a \cdot u \right)
    - \epsilon_\alpha^{\alpha} \mu(u) \nabla \mu(u) = 0, \\[1.5ex]
    \Delta a
    - \epsilon_\alpha^{1-\alpha} f_{\alpha, \epsilon_\alpha} u^* \nabla_A u
    + \langle d a, a \rangle
    + \langle a, [a, a] \rangle
    - \left[ F_A \lh a^\sharp \right] = 0,
  \end{cases}
\end{equation}
where \( f_{\alpha, \epsilon_\alpha} = \alpha \left( \epsilon_\alpha + |\nabla_A u|^2 \right)^{\alpha-1} \).

The following proposition provides a small energy estimate for generalized \( \alpha \)-YMH fields. The proof follows closely that of \cref{prop:small-energy-estimate}; see also \cite{AiSongZhu2019boundary}*{Cor.~1}.

\begin{prop}[Rescaled small energy estimates; \cite{AiSongZhu2019boundary}]\label{prop:small-energy-estimates-scaled}
  Suppose \( (A,u) \) is a smooth generalized \( \alpha \)-YMH field on the unit disc \( D \subset \mathbb{R}^2 \) satisfying \eqref{eq:generalized-alpha-YMH-local}, with \( \lim_{\alpha\to1}\epsilon_\alpha^{\alpha-1}\in(\beta_0,1] \) for some constant \( \beta_0>0 \), and \( \|F_A\|_{L^2(D)} \leq \delta_0 \), where \( \delta_0 \) is the uniform constant from the existence theorem for the Coulomb gauge. Then there exist constants \( \epsilon_0 > 0 \) and \( \alpha_0 > 1 \) such that, for all \( 1 < \alpha < \alpha_0 \) and all \( 1< p < +\infty \), if
  \[
    \|\nabla_A u\|_{L^2(D)} \leq \epsilon_0,
  \]
  the following estimates hold for any \( k = 2, 3, \ldots \):
  \[
    \|u - \bar{u}\|_{W^{k,p}(D_{1/2})} \leq C \big( \|\nabla_A u\|_{L^2(D)} + \|F_A\|_{L^2(D)} \big),
  \]
  and
  \[
    \|a\|_{W^{k,p}(D_{1/2})} \leq C \big( \|\nabla_A u\|_{L^2(D)} + \|F_A\|_{L^2(D)} \big),
  \]
  where \( A = d + a \) locally, \( \bar{u} \) denotes the average of \( u \) over \( D \), and \( C \) is a constant depending only on \( \delta_0 \), \( \|\mu\|_{W^{1,\infty}} \), \( p \), \( k \), \( \alpha_0 \), and \( \epsilon_0 \).
\end{prop}
With the small energy estimates of \cref{prop:small-energy-estimate}, items \eqref{item:prop-small-energy-I} and \eqref{item:prop-small-energy-II} in \cref{prop:bubbing} follow directly. Near each energy concentration point \( x^j \in \mathcal{S} \), we perform the rescaling \( (A,u) \mapsto (A_\lambda, u_\lambda) \) with \( \lambda = r_\alpha^j \), and apply the rescaled small energy estimates \cref{prop:small-energy-estimates-scaled}. This yields item \eqref{item:prop-small-energy-III} in \cref{prop:bubbing}.
\subsection{Analytic lemmas in Lorentz spaces}
In this subsection, we collect several analytic results, particularly elliptic estimates in Lorentz spaces, which are essential for the neck analysis in our compactness arguments. Recall that for a measurable function \( f \) on \( \Omega \subset \mathbb{R}^m \), the distribution function and the non-increasing rearrangement are defined as
\[
  \lambda_f(s) := \left| \left\{ x \in \Omega : |f(x)| \geq s \right\} \right|, \qquad
  f^*(t) := \inf \left\{ s \geq 0 : \lambda_f(s) \leq t \right\},
\]
respectively. One easily checks that
\[
  \left| \left\{ t > 0 : f^*(t) \geq s \right\} \right| = \lambda_f(s),
\]
so the measure of \( \{t > 0 : f^*(t) \geq s\} \) equals the measure of \( \{x \in \Omega : |f(x)| \geq s\} \).

The Lorentz space \( L^{p,q} \) of measurable function $f \mathpunct{:} \Omega\to \mathbb{R}$, where \( 0 < p, q < +\infty \) and $\Omega$ is an open subset in $\mathbb{R}^m$, is defined by
\[
  L^{p,q}(\Omega) := \left\{ f : \|f\|_{L^{p,q}} := \left( \frac{q}{p} \int_0^\infty \left( t^{1/p} f^*(t) \right)^q \frac{dt}{t} \right)^{1/q} < +\infty \right\},
\]
where we normalize so that \( \|\chi_{[0,1]}\|_{L^{p,q}} = 1 \). When \( q = \infty \), the Lorentz space \( L^{p, \infty} \) is defined by
\[
  L^{p,\infty}(\Omega) := \left\{ f : \|f\|_{L^{p,\infty}} = \sup_{t > 0} t^{1/p} f^*(t) < +\infty \right\}.
\]
In general, \( \|\cdot\|_{L^{p,q}} \) does not define a norm, except when \( 1 \leq q \leq p \), in which case it is indeed a norm. For \( 1 < q' < q'' \), the following inclusions hold:
\[
  L^{p,1} \subset L^{p,q'} \subset L^{p,q''} \subset L^{p,\infty}, \qquad L^{p,p} = L^p.
\]
Moreover, if \( |\Omega| < +\infty \), then for any \( q \) and \( q' \), we have \( L^{p',q'} \subset L^{p,q} \) whenever \( p < p' \). Finally, for \( 1 < p < +\infty \) and \( 1 \leq q \leq +\infty \), the dual space of \( L^{p,q} \) is given by \( L^{\frac{p}{p-1}, \frac{q}{q-1}} \).
\begin{lem}[\cite{LiZhu2019Energy}*{Lem.~2.4}]\label{lem:div-elliptic-estimate}
  Suppose \( \vec{F} \) is a vector field supported in \( D \) and \( u \) solves
  \[
    \begin{cases}
      \Delta u = \div \vec{F}, & x \in D, \\
      u = 0, & x \in \partial D,
    \end{cases}
  \]
  then the following elliptic estimate holds:
  \[
    \|\nabla u\|_{L^{2,q}(D)} \leq C_q \|\vec{F}\|_{L^{2,q}(D)}, \quad \forall\, 1 \leq q \leq \infty.
  \]
\end{lem}
\begin{rmk}
  For \( q=2 \), we have \( C_2=1 \), see \cite{LiZhu2019Energy}*{Lem.~2.3}.
\end{rmk}

The next lemma is a standard elliptic \( L^p \) estimate together with H\"older's inequality for Lorentz spaces.
\begin{lem}\label{lem:elliptic-estimate}
  Assume \( F \in L^p(D) \) for some \( p > 1 \), and \( u \) solves
  \[
    \begin{cases}
      \Delta u = F, & x \in D, \\
      u = 0, & x \in \partial D,
    \end{cases}
  \]
  then
  \[
    \|\nabla u\|_{L^{2,1}(D)} \leq C \|F\|_{L^p(D)}.
  \]
\end{lem}

The following lemma is a classical result in harmonic analysis. It asserts that for the Poisson equation with a right-hand side of Jacobian structure, one gains additional regularity---a phenomenon known as compensation.
\begin{lem}[\citelist{\cite{Wente1969Existence}\cite{CoifmanLionsMeyerSemmes1993Compensated}\cite{Helein2002Harmonic}*{Sect.~3.2 and 3.3}}]\label{lem:Wente-L21}
  Suppose \( f,g \in W^{1,2}(\mathbb{R}^2) \), and \( u \in L^1(\mathbb{R}^2) \) solves
  \[
    \Delta u = \nabla f \cdot \nabla^\perp g,
  \]
  then \( \nabla f \cdot \nabla^\perp g \) belongs to the Hardy space \( \mathcal{H}^1(\mathbb{R}^2) \), and
  \[
    \|\nabla u\|_{L^{2,1}} \leq C \|\nabla f \cdot \nabla^\perp g\|_{\mathcal{H}^1} \leq C \|\nabla f\|_{L^2} \|\nabla g\|_{L^2}.
  \]
\end{lem}

We also record the following variant, due to F. Bethuel, which is particularly useful in the presence of Lorentz space regularity.
\begin{lem}[\citelist{\cite{Helein2002Harmonic}*{Thm.~3.4.5}\cite{LiLiuZhuZhu2021Energy}*{Lem.~2.4}}]\label{lem:Wente-Bethuel-Helein}
  Suppose \( f,g \in W_0^{1,2}(D) \) and \( \nabla f \in L^{2,\infty}(D) \). If \( u \in W_0^{1,2}(D) \) solves
  \[
    \Delta u = \nabla f \cdot \nabla^\perp g,
  \]
  then
  \[
    \|\nabla u\|_{L^2(D)} \leq C \|\nabla f\|_{L^{2,\infty}(D)} \|\nabla g\|_{L^2(D)}.
  \]
\end{lem}
Finally, the Lorentz space \( L^{2,1} \) embeds continuously into \( C^0 \). 
\begin{lem}[\citelist{\cite{Helein2002Harmonic}*{Thm.~3.3.4}\cite{LiLiuZhuZhu2021Energy}*{Lem.~2.5}}]\label{lem:embedding-L21-C0}
  Suppose \( u\in W^{1,2}(\mathbb{R}^2) \) vanishes at infinity and \( \nabla u \in L^{2,1}(\mathbb{R}^2) \). Then \( u \in C^0(\mathbb{R}^2) \), and moreover,
  \[
    \|u\|_{C^0} \leq C \|\nabla u\|_{L^{2,1}}.
  \]
\end{lem}
Here and in the sequel, we omit the domain in various norms when it is \( \mathbb{R}^2 \). The symbol \( C \) denotes a positive constant, which may depend on \( \Lambda \), \( K \), \( \epsilon_0 \), and \( \delta_0 \), and may vary from one inequality to another.
\section{The Euler--Lagrange equations, the conservation law and the Pohozaev-type equality}\label{sec:EL-ConservationLaw-Pohozaev}
In this section, we first explicitly compute the Euler--Lagrange equation \eqref{eq:alpha-YMH-local} in the case where the fiber \( F \) is the sphere \( S^{K-1} \subset \mathbb{R}^K \). We then derive a new conservation law for \( \alpha \)-YMH fields, and separately establish a Pohozaev-type identity for \( \alpha \)-YMH fields. These preparations are essential for establishing the compactness of sequences of \( \alpha \)-YMH fields with uniformly bounded energy.
\subsection{The Euler--Lagrange equation for \texorpdfstring{\( \alpha \)}{α}-YMH fields}
Suppose \( \mathcal{P} \) is a \( G \)-principal bundle over a closed Riemannian manifold \( \Sigma \), and let \( \mathcal{F} = \mathcal{P} \times_G F \) denote the associated fiber bundle, where \( F \) is a Riemannian manifold equipped with an isometric \( G \)-action. Let \( A \) be a connection on \( \mathcal{P} \), and \( u \) a local section of \( \mathcal{F} \) over some open set \( U \subset \Sigma \). The induced covariant derivative is denoted by \( \nabla_A \), which locally takes the form \( \nabla_A = d + a \) for some \( \mathfrak{g} \)-valued \( 1 \)-form \( a \). Explicitly,
\[
  \nabla_A u = d u + a \cdot u,
\]
where ``\( \cdot \)'' denotes the infinitesimal action of \( \mathfrak{g} \) on \( F \). More precisely, for \( a = a_\beta dx^\beta \in \Omega^1(\mathfrak{g}) = \Gamma(T^*\Sigma \otimes \mathfrak{g}) \), let \( \varphi_s = \exp(sa_\beta) \) be the one-parameter family of bundle automorphisms generated by \( a_\beta \). Then \( a_\beta \) induces a fundamental vector field \( X_{a_\beta} \in \Gamma(TF) \) at \( u \) via
\[
  X_{a_\beta}(u) = \left. \frac{d}{ds} \right|_{s=0} \varphi_s(u) =: a_\beta \cdot u.
\]
We define 
\[
  a\cdot u = X_a(u) \mathpunct{:}= a_\beta \cdot u\, dx^\beta.
\]
Similarly, for a vector field \( V \in \Gamma(TF) \),
\[
  \nabla_V X_{a_\beta} = \left. \frac{\nabla}{ds} \right|_{s=0} (\varphi_s)_* V =: a_\beta \cdot V, \qquad
  a\cdot V \mathpunct{:}= \nabla_V X_a \mathpunct{:}= \nabla_V \left( X_{a_\beta}\, dx^\beta \right).
\]
where \( \frac{\nabla}{ds} \) denotes the covariant derivative along the curve \( \varphi_s \), induced by the Levi-Civita connection on \( F \). 

Moreover, since \( G \) is a connected, compact Lie group acting isometrically on \( F \), a general equivariant embedding theorem due to Moore and Schlafly \cite{MooreSchlafly1980equivariant} guarantees the existence of an orthogonal representation \( \rho: G \to SO(K) \) and an isometric embedding \( i: F \hookrightarrow \mathbb{R}^K \) such that \( i(g.y) = \rho(g)\, i(y) \) for all \( y \in F \) and \( g \in G \). Consequently, the Lie algebra \( \mathfrak{g} \) can be regarded as a subalgebra of \( \mathfrak{so}(K) \), the space of skew-symmetric \( K \times K \) matrices; the infinitesimal action of \( a \) on \( y \in F \) is given by matrix multiplication:
\[
  a.y = X_a(y) = \rho(a) \cdot y =: \chi_a \cdot y,
\]
and the action of \( a \) on a vector field \( V \in \Gamma(TF) \) is
\[
  a.V = \nabla_V X_a = (\chi_a \cdot V)^\top = \chi_a \cdot V - \mathrm{I\!I}(y)(X_a, V),
\]
where \( (\cdot)^\top \) denotes the orthogonal projection from \( \mathbb{R}^K \) onto the tangent space \( T_y F \), and \( \mathrm{I\!I} \) is the second fundamental form of the embedding \( F \hookrightarrow \mathbb{R}^K \).

In particular, for \( F = S^{K-1} \subset \mathbb{R}^K \) and \( G \subset \mathrm{SO}(K) \), we have \( \mathfrak{g} \subset \mathfrak{so}(K) \), i.e., the space of all skew-symmetric \( K \times K \) matrices, and \( X_a(u) = a u \) by matrix multiplication (where \( \rho \) is the inclusion). Furthermore, since \( \mathrm{I\!I}(u)(X, Y) = -\langle X, Y \rangle u \), it follows that
\[
  a \cdot V = a V - \langle a u, V \rangle u.
\]

\begin{prop}\label{prop:EL-local}
  Suppose \( (A, \phi) \in \mathscr{A} \times \mathscr{S} \) is an \( \alpha \)-YMH field. Under a local trivialization, \( (A, \phi) \) is represented by \( A = d + a \) and \( \phi = u \). If the fiber \( F = S^{K-1} \subset \mathbb{R}^K \) and \( G \subset \mathrm{SO}(K) \), then the Euler--Lagrange equation \eqref{eq:alpha-YMH} reduces to \eqref{eq:alpha-YMH-local}.
\end{prop}

\begin{proof}
Since the typical fiber \( F \) is the sphere \( S^{K-1} \subset \mathbb{R}^K \), and \( G\subset\mathrm{SO}(K) \). Locally, we may write
\[
  \nabla_A u = du + a \cdot u = \left( \partial_\alpha u^i + a_{\alpha, j}^i u^j \right) dx^\alpha \otimes \epsilon_i, \quad a_\alpha = (a_{\alpha, j}^i) \in \mathfrak{so}(K),
\]
where \( \{ \epsilon_i \}_{i=1}^K \) denotes the standard basis of \( \mathbb{R}^K \). Thus,
\begin{align*}
  \left| \nabla_A u \right|^2_h
  &= g^{\alpha\beta}(u) \left( \partial_\alpha u^i + a_{\alpha, k}^i u^k \right) \left( \partial_\beta u^i + a_{\beta, l}^i u^l \right) \\
  &= g^{\alpha\beta} \partial_\alpha u^i \partial_\beta u^i
  + 2 g^{\alpha\beta} \partial_\alpha u^i a_{\beta, l}^i u^l
  + g^{\alpha\beta} a_{\alpha, k}^i a_{\beta, l}^i u^k u^l.
\end{align*}
To simplify the presentation, we assume that the metric on \( U \subset \Sigma \) is Euclidean. The general case can be treated analogously using normal coordinates. Then,
\[
  \left| \nabla_A u \right|^2 = \partial_\alpha u^i \partial_\alpha u^i
  + 2 \partial_\alpha u^i a_{\alpha, l}^i u^l
  + a_{\alpha, k}^i a_{\alpha, l}^i u^k u^l.
\]

The Euler--Lagrange equation can be derived as follows; we refer the reader to \cite{Evans2010Partial} for the case of harmonic maps. For clarity, we first consider the case \( \alpha=1 \) and fix the connection \( A \), i.e., we consider variations only in the map \( u \). Let \( u_t = \frac{u + t w}{|u + t w|} \in \mathbb{R}^K \) be a variation through maps into the sphere. Note that \( u_0'=w-\left\langle u,w \right\rangle u \). Since \( |u|^2 = 1 \), it follows that \( \partial_\alpha u^i u^i = 0 \). Moreover, recall that \( a_{\alpha, i}^l + a_{\alpha, l}^i = 0 \) since \( a_\alpha \) is skew-symmetric. Therefore,
\begin{align*}
  &\left. \frac{d}{dt} \right|_{t=0} \mathcal{L}(A, u_t) - \int_\Sigma \langle \mu(u)\nabla \mu(u), w \rangle \\
  &\qquad = \frac{1}{2} \int_\Sigma \left. \frac{d}{dt} \right|_{t=0} |\nabla_A u_t|^2_{h(u_t)} 
  = \int_\Sigma \langle \nabla_A u_0, \nabla_A u_0' \rangle_{\mathbb{R}^K} \\
  &\qquad = \int_\Sigma \left( \partial_\alpha u^i + a_{\alpha, k}^i u^k \right) \left( \partial_\alpha (w^i - u^j w^j u^i) + a_{\alpha, l}^i (w^l - u^j w^j u^l) \right) \\
  &\qquad = \int_\Sigma \nabla u^i \nabla w^i - \partial_\alpha u^i \partial_\alpha (u^i u^j w^j) + a_{\alpha, k}^i a_{\alpha, l}^i u^k (w^l - u^j u^l w^j) \\
  &\qquad\qquad + \int_\Sigma \partial_\alpha u^i a_{\alpha, l}^i (w^l - u^j w^j u^l) + a_{\alpha, k}^i u^k \partial_\alpha (w^i - u^j w^j u^i).
\end{align*}
Upon rearranging and integrating by parts, and using skew-symmetry and orthogonality relations, we obtain
\begin{align*}
  &\left. \frac{d}{dt} \right|_{t=0} \mathcal{L}(A, u_t) - \int_\Sigma \langle \mu(u)\nabla \mu(u), w \rangle \\
  &\qquad=\int_\Sigma-\div \nabla  u^iw^i-\underbrace{\partial_\alpha u^iu^i}_{=0}\partial_\alpha (w^j u^j)-\partial_\alpha u^i\partial_\alpha u^i w^ju^j+a_{\alpha,k}^ia_{\alpha,l}^i u^kw^l- a_{\alpha,k}^iu^ka_{\alpha,l}^iu^l u^jw^j\\
  &\qquad\qquad+\int_\Sigma \partial_\alpha u^ia_{\alpha,l}^iw^l-\partial_\alpha u^ia_{\alpha,l}^i u^ju^l w^j+a_{\alpha,k}^iu^k\partial_\alpha w^i-a_{\alpha,k}^iu^k\partial_\alpha(u^iu^jw^j)
\end{align*}
These terms can be grouped and expressed as
\begin{align*}
  &\int_\Sigma -\left\langle \div \nabla u+\lvert \nabla  u \rvert^2 u,w \right\rangle
  -\int_\Sigma \left\langle a\cdot a u +\lvert au \rvert^2u+a\cdot du-\left\langle a\cdot du,u \right\rangle u, w \right\rangle\\
  &\qquad +\int_\Sigma \left\langle -a\cdot du+d^*a\cdot u,w \right\rangle+\left\langle \left\langle a\cdot du-d^*a\cdot u, u \right\rangle u,w \right\rangle,
\end{align*}
where we used \( [a u]^i = a_{\alpha, k}^i u^k dx^\alpha \). Observe that
\begin{align*}
  \int_M a_{\alpha,k}^iu^k\partial_\alpha w^i
  &=\int_M\left\langle [au]^i,d w^i \right\rangle_g
  =\int_Md^*[au]^i w^i
  =\int_M\left\langle -a\cdot du+d^*a\cdot u,w \right\rangle,
\end{align*}
since
\begin{align*}
  d^*[au]^i&=d^*\left( a_{\alpha,k}^iu^kdx^\alpha \right)
  =- \partial_\alpha(a_{\alpha,k}^iu^k)
  =-a_{\alpha,k}^i\partial_\alpha u^k- \partial_\alpha a_{\alpha,k}^iu^k\\
           &=-a_{\alpha,k}^i\partial_\alpha u^k+d^*a^i_ku^k
           =-[a\cdot du]^i+[d^*a\cdot u]^i.
\end{align*}
That is,
\[
  d^*(au)=-a\cdot du+d^*a\cdot u,
\]
which generalizes the standard formula for scalar functions \( f \):
\[
  d^*(fa)=- a\cdot df+fd^*a,\quad f\in C^\infty(\Sigma).
\]

Proceeding, we obtain
\begin{align*}
  \left. \frac{d}{dt} \right\rvert_{t=0}\mathcal{L}(A,u_t)
  &=-\int_\Sigma\left\langle \div \nabla  u+\lvert \nabla u \rvert^2u+a\cdot au+\lvert au \rvert^2 u+2a\cdot du-\left\langle a\cdot du,u \right\rangle u,w \right\rangle\\
  &\qquad +\int_\Sigma\left\langle d^*a\cdot u-\left\langle d^*a\cdot u,u \right\rangle u+\left\langle a\cdot du, u\right\rangle u+\mu \nabla \mu,w \right\rangle\\
  &=-\int_\Sigma\left\langle \div \nabla  u+\lvert \nabla_A u \rvert^2 u+a\cdot a u+2a\cdot du-d^*a\cdot u-\mu \nabla \mu,w \right\rangle,
\end{align*}
where we used
\[
  \lvert \nabla _Au \rvert^2=\lvert \nabla u \rvert^2+\lvert a u \rvert^2-2\left\langle a\cdot du,u \right\rangle,
  \quad 
  \left\langle d^*a\cdot u,u \right\rangle=-\partial_\alpha a_{\alpha,i}^ju^iu^j=0,
\]
since \( a_\alpha \) is skew-symmetric. Thus, the Euler--Lagrange equation for the map component in \( \mathcal{L}(A, u) \) is
\[
  \Delta u+\lvert \nabla _Au \rvert^2 u-d^*a\cdot u+2a\cdot du+a\cdot a u-\mu(u)\nabla \mu(u)=0.
\]
Note also that
\[
  \div(\nabla _Au^i)=\partial_\alpha \left( \partial_\alpha u^i+a_{\alpha j}^iu^j \right)
  =\Delta u^i-[d^*a\cdot u]^i+[a\cdot du]^i,
\]
so we can rewrite the Euler--Lagrange equation as
\[
  \div(\nabla _Au)+\lvert \nabla _Au \rvert^2 u+a\cdot du+a\cdot a u-\mu(u)\nabla \mu(u)=0.
\]

For the \( \alpha \)-YMH fields, the Euler--Lagrange equation for the map component is derived similarly. Note that from the previous calculation,
\[
  -\nabla _A^*\nabla _Au=\div(\nabla _Au)+\lvert \nabla _Au \rvert^2 u+a \cdot du+a\cdot a u,
\]
so by \eqref{eq:alpha-YMH},
\[
  0=\nabla _A^*\left( f_\alpha \nabla _Au \right)+\mu(u)\nabla \mu(u)=-\left\langle df_\alpha,\nabla _Au \right\rangle+f_\alpha\nabla _A^*\nabla _Au+\mu(u)\nabla \mu(u),
\]
where \( f_\alpha=\alpha\left( 1+\lvert \nabla_Au \rvert^2\right)^{\alpha-1} \). Therefore,
\begin{align*}
  \div(f_\alpha \nabla _Au)&=\left\langle df_\alpha,\nabla _Au \right\rangle+f_\alpha\div(\nabla _Au)
  =f_\alpha\left( \div(\nabla _Au)+\nabla _A^*\nabla _Au \right)+\mu(u)\nabla \mu(u)\\
                           &=-f_\alpha\left( \lvert \nabla _Au \rvert^2u+a\cdot du+a\cdot a u \right)+\mu(u)\nabla \mu(u).
\end{align*}
That is
\[
  \div(f_\alpha\nabla _Au)+f_\alpha\left( \lvert \nabla _Au \rvert^2u+a\cdot du+a\cdot a u \right)-\mu(u)\nabla \mu(u)=0,
\]
which is exactly the first equation in \eqref{eq:alpha-YMH-local}.

Finally, for the connection component, recall that locally \( \nabla_A = d + a \) and
\[
  F_A=da+a\wedge a,
\]
with \( D_A F_A = 0 \) by the second Bianchi identity. A direct calculation yields
\[
  D_A^*F_A=d^*da-[a,d^*a]-\langle da,a \rangle-\langle a,[a,a] \rangle.
\]
Therefore, the equation for the connection \( A \) is
\[
  d^*da-[a,d^*a]-\langle da,a \rangle-\langle a,[a,a] \rangle+f_\alpha u^*\nabla _Au=0.
\]
In particular, in a local Coulomb gauge, i.e., when \( d^* a = 0 \), the local equations for the \( \alpha \)-YMH fields reduce to the second equation of \eqref{eq:alpha-YMH-local}, where we apply the Weitzenb\"ock formula:
\[
  \Delta_H a=d^*da=-\Delta a+a\lh \mathrm{Ric}_\Sigma+[F_A\lh a^\sharp ],\quad \Delta=-\nabla ^*\nabla.
\]
\end{proof}

\subsection{The conservation law}
We now derive a conservation law for \( \alpha \)-YMH fields in the setting of a spherical bundle. Since \( G \subset \mathrm{SO}(K) \) represents a global symmetry of the \( \alpha \)-YMH theory—depending only on the intrinsic geometry of the sphere—the vector field \( X_b \) generated by \( b\in \mathfrak{so}(K) \) is an infinitesimal symmetry of the Lagrangian \( L_\alpha \). More precisely, for \( g \in G \subset \mathrm{SO}(K) \), the symmetry action on a field configuration \( (a, u) \) is given by
\[
  g \cdot (a, u) = \big( gag^{-1} + g d g^{-1},\, gu \big).
\]
Under this action, it is straightforward to verify that the Lagrangian density
\[
  L_\alpha(x, (a, u), d(a, u)) = \sum_{i, \alpha} \left( 1 + |du + a u|^2 \right)^\alpha - 1 + |da + a \wedge a|^2 + |\mu(u)|^2
\]
is invariant. Consequently, Noether's theorem yields the following abstract form of the conservation law for \( \alpha \)-YMH fields. 
\begin{prop}\label{prop:abstract-conservation-law}
  Suppose \( (A,u) \) is a critical point of the local \( \alpha \)-YMH functional over \( B^m \)
  \[
    \mathcal{L}_\alpha(A,u;B^m)=\int_{B^m}L_\alpha(x,(a,u),d(a,u)),
  \]
  where \( A=d+a \) and \( u \mathpunct{:}B^m\to S^{K-1} \). Then, for any \( b\in \mathfrak{so}(K) \), we have 
\begin{equation}\label{eq:conservation-law}
  \div\left[ \left( \frac{\partial L_\alpha}{\partial P_{\gamma\beta,j}^i}[b, a_\gamma]_j^i
  +\frac{\partial L_\alpha}{\partial Q_\beta^i} b_{j}^iu^j \right)\partial_{x^\beta} \right]=0,
\end{equation}
where the Lagrangian \( L_\alpha(x,(a,u),d(a,u)) \) is locally expressed as
\[
  \begin{multlined}
    L_\alpha(x,(y,z),(P,Q))\\
    =\sum_{i,\alpha}\left(  1+ \left( Q_\alpha^i \right)^2+2Q_\alpha^i y_{\alpha,j}^iz^j+y_{\alpha,j}^iz^j y_{\alpha,k}^iz^k \right)^\alpha-1+\lvert P+y\wedge y \rvert^2+\lvert \mu(z) \rvert^2.
  \end{multlined}
\]
\end{prop}
For the reader's convenience, we briefly sketch the derivation of \eqref{eq:conservation-law}, following the presentation of Noether's theorem in \cite{Helein2002Harmonic}*{Sect.~1.3.1}.
\begin{proof}
The infinitesimal generator \( X_b \) of the symmetry is given by
\[
  X_{b}(y,z)=\left. \frac{d}{ds} \right\rvert_{s=0}\exp(s b)\cdot (y,z)= \left(-\left( db+[y,b] \right), bz\right).
\]
For constant \( b \in \mathfrak{so}(K) \), this becomes \( X_b(a,u) = ([b,a],\, b u) \). To check that \( X_b \) is an infinitesimal symmetry of \( L_\alpha(x,(a,u),d(a,u)) := (1+|\nabla_A u|^2)^\alpha - 1 + |F_A|^2 + |\mu(u)|^2 \) at a critical point \( (A,u) \), we note:
\begin{itemize}
  \item \(  \nabla _{g\cdot A} (g\cdot u) = d(gu)+a_g (gu) = dg u+gdu+gdg^{-1}gu+gag^{-1}gu = g(du+au)=g\nabla _Au \);
  \item \( F_{g\cdot A}=da_g+a_g\wedge a_g=gdag^{-1}+ga\wedge  ag^{-1}=gF_Ag^{-1} \);
  \item \( \mu(g\cdot u)=\mu(u) \).
\end{itemize}

Let \( \varphi \in C_0^\infty(B^m, \mathbb{R}) \) be a test function. Since \( (A,u) \) is a critical point of \( \mathcal{L}_\alpha \), we have
\begin{equation}\label{eq:critical-point}
  \mathcal{L_\alpha}\left( (A,u)+s\varphi X_b(a,u)+o(s) \right)=\mathcal{L_\alpha}(A,u)+o(s).
\end{equation}
A Taylor expansion yields
\begin{align*}
  &\mathcal{L}_\alpha((A,u)+s\varphi X_b(a,u)+o(s))\\
  &\qquad=\int_{B^m} L\left(x,(a,u)+s\varphi X_b(a,u),d\left( a,u \right)+s d\left(\varphi X_b(a,u)\right)\right)+o(s)\\
  &\qquad=\int_{B^m} L\left(x,(a,u)+s\varphi X_b(a,u),d\left( a,u \right)+s \varphi d X_b(a,u)+sd\varphi X_b(a,u) \right)+o(s)\\
  &\qquad=\int_{B^m} L\left(x,(a,u)+s\varphi X_b(a,u),d(a,u)+s\varphi dX_b(a,u)\right)+o(s)\\
  &\qquad\qquad +\int_{B^m} s\left\{ \frac{\partial L}{\partial P}[b,a]+\frac{\partial L}{\partial Q} bu \right\} d\varphi.
\end{align*}
Since the Lagrangian is invariant under the symmetry, 
\[
  L\left(x,\exp(t b)\cdot(a,u),d\left( \exp(t b)\cdot(a,u) \right)\right)=L(x,(a,u),d(a,u)),
\]
and so
\[
  L(x,(a,u)+tX_b(a,u),d(a,u)+tdX_b(a,u))=L(x,(a,u),d(a,u))+o(t),
\]
which implies
\[
  L\left(x,(a,u)+s\varphi X_b,d(a,u)+s\varphi dX_b\right)=L(x,(a,u),d(a,u))+o(s).
\]
Then, we obtain
\begin{align*}
  &\mathcal{L}_\alpha((A,u)+s\varphi X_b(a,u)+o(s))\\
  &\qquad=\int_{B^m} L\left( x,(a,u),d(a,u) \right)+o(s)
  +\int_{B^m} s\left\{ \frac{\partial L}{\partial P}[b,a]+\frac{\partial L}{\partial Q} bu \right\} d\varphi\\
  &\qquad=\mathcal{L}_\alpha(A,u)+o(s)
    +\int_{B^m} s\left\{ \frac{\partial L}{\partial P}[b,a]+\frac{\partial L}{\partial Q} bu \right\} d\varphi.
\end{align*}
By \eqref{eq:critical-point}, this yields
\[
  \int_{B^m}\left( \frac{\partial L}{\partial P_{\gamma\beta,j}^i}[b,a_\gamma]^i_j+\frac{\partial L}{\partial Q_\beta^i}b_j^iu^j\right)\frac{\partial\varphi}{\partial x^\beta}=0,
\]
which leads to \eqref{eq:conservation-law}, i.e.,
\[
  \div\left[ \left( \frac{\partial L}{\partial P_{\gamma\beta,j}^i}[b,a_\gamma]^i_j+\frac{\partial L}{\partial Q_\beta^i}b_j^iu^j\right)\partial_{x^\beta} 
\right]=0.
\]
Or more concisely,
\[
  \div\left( \frac{\partial L(x,(y,z),(P,Q))}{\partial (P,Q)}\cdot X_b(a,u) \right)=0.
\]
\end{proof}
We now explicitly expand the abstract conservation law stated in \cref{prop:abstract-conservation-law}, thereby obtaining the concrete form needed in this paper. This will complete the proof of \cref{thm:conservation-law}.
\begin{proof}[Proof of \cref{thm:conservation-law}]
Recall that \( a = a_\alpha dx^\alpha \) with \( a_\alpha \in \mathfrak{so}(K) \). The curvature \( 2 \)-form is given by
  \[
    F_A=\frac{1}{2}F_{\alpha\beta}dx^\alpha\wedge dx^\beta,
  \]
  where \( F_{\alpha\beta}=-F_{\beta\alpha}=\partial_\alpha a_\beta-\partial_\beta a_\alpha+[a_\alpha,a_\beta] \). The squared norm of \( F_A \) is then
  \[
    \lvert F_A \rvert^2=\frac{1}{2}F_{\alpha\beta,j}^iF_{\beta\alpha,i}^j.
  \]

  Using the identity
  \[
    F_{\alpha\beta}=P_{\alpha\beta}-P_{\beta\alpha}+[y_\alpha,y_\beta],
  \]
  we compute the relevant variational derivatives:
  \[
    \frac{\partial F_{\eta\xi,k}^l}{\partial P_{\gamma\beta,j}^i}
    =\delta_{\eta\gamma}\delta_{\xi\beta}\delta_{kj}\delta_{il}-\delta_{\xi\gamma}\delta_{\eta\beta}\delta_{kj}\delta_{il}
    =\left( \delta_{\eta\gamma}\delta_{\xi\beta}-\delta_{\xi\gamma}\delta_{\eta\beta} \right)\delta_{kj}\delta_{il}.
  \]
  Therefore,
  \begin{align*}
    -\frac{\partial L_\alpha}{\partial P_{\gamma\beta,j}^i}  [a_\gamma,b]_j^i 
  &=-\frac{1}{2}\left( \frac{\partial F_{\eta\xi,k}^l}{\partial P_{\gamma\beta,j}^i}F_{\xi\eta,l}^k+
  F_{\eta\xi,k}^l \frac{\partial F_{\xi\eta,l}^k}{\partial P_{\gamma\beta,j}^i}\right) [a_\gamma,b] _j^i\\
  &=\frac{1}{2}\left( F_{\beta\gamma,i}^j-F_{\gamma\beta,i}^j+F_{\beta\gamma,i}^j-F_{\gamma\beta,i}^j \right)\left[ b,a_\gamma \right]_j^i\\
  &=\left( F_{\beta\gamma,i}^j-F_{\gamma\beta,i}^j \right)\left[ b,a_\gamma \right]_j^i
  =2F_{\beta\gamma,i}^j[b,a_\gamma]_j^i.
  \end{align*}

  Similarly, we have
  \begin{align*}
    \frac{\partial L_\alpha}{\partial Q_\beta^i}b_j^iu^j
  &=f_\alpha \frac{\partial \left( 
      Q_\gamma^kQ_\gamma^k+2Q_\gamma^k y_{\gamma,l}^kz^l+y_{\gamma,l}^kz^ly_{\gamma,k}^kz^k
  \right)}{\partial Q_\beta^i}b_j^iu^j\\
  &=f_\alpha \left( 2Q_\gamma^k\left( \delta_{\beta\gamma}\delta_{ki} \right)+2\delta_{\beta\gamma}\delta_{ik}y_{\gamma,l}^kz^l \right) b_j^iu^j\\
  &=2f_\alpha \left( Q_\beta^i+y_{\beta,l}^iz^l \right)b_j^iu^j
  =2f_\alpha\left( \partial_\beta u^i+a_{\beta,l}^iu^l \right)b_j^iu^j,
  \end{align*}
  where \( f_\alpha=\alpha(1+\lvert \nabla _Au \rvert^2)^{\alpha-1} \).

  Putting these calculations together, the conservation law takes the form
  \[
    \div\left[ \left( f_\alpha [\nabla _Au]_\beta^pb_q^pu^q 
    +F_{\gamma\beta,r}^q\left( a_{\gamma,p}^rb_q^p-b^r_pa_{\gamma,q}^p \right) \right)\partial_{x^\beta} \right]=0,\quad b=(b_q^p)\in \mathfrak{so}(K),
  \]
  which is precisely \eqref{eq:conservation-law-general-b}, and this completes the proof of \cref{thm:conservation-law}.
\end{proof}
We now proof the component form of conservation law.
\begin{proof}[Proof of \cref{cor:conservation-law-component}]
  Now, let \( b = (b_q^p) \) be the standard basis element of \( \mathfrak{so}(K) \) defined by
  \[
    b_q^p=\begin{cases}
      1,&(p,q)=(i,j),\\
      -1,&(p,q)=(j,i),\\
      0,&\text{otherwise}.
    \end{cases} 
  \]
  Substituting this choice of \( b \) into \eqref{eq:conservation-law-general-b}, we obtain
  \begin{align*}
    0&=\div\left[ \left( f_\alpha\left( [\nabla _Au]_\beta^iu^j-[\nabla _Au]_\beta^ju^i \right)+F_{\gamma\beta,r}^j a_{\gamma,i}^r-F_{\gamma\beta,i}^ra_{\gamma,r}^j-F_{\gamma\beta,r}^ia_{\gamma,j}^r+F_{\gamma\beta,j}^ra_{\gamma,r}^i  \right)\partial_{x^\beta} \right]\\
     &=\div\left[ \left( f_\alpha\left( [\nabla _Au]_\beta^iu^j-[\nabla _Au]_\beta^ju^i \right)
     -([F_{\gamma\beta}, a_\gamma]-[F_{\gamma\beta},a_\gamma]^T)^i_j \right)\partial_{x^\beta} \right]\\
     &=\div\left[ \left( f_\alpha\left( [\nabla _Au]_\beta^iu^j-[\nabla _Au]_\beta^ju^i \right)-2\left[ F_{\gamma\beta},a_\gamma \right]_j^i\right) \partial_{x^\beta}\right],
      \end{align*}
      where we used the identities
      \[
        F_{\beta\gamma}^T=-F_{\beta\gamma},\quad a_\gamma^T=-a_\gamma\implies 
        \left[ F_{\beta\gamma},a_\gamma \right]^T=-\left[ F_{\beta\gamma},a_\gamma \right].
      \]
      This yields \eqref{eq:conservation-law-component}, and thus \cref{cor:conservation-law-component} is proved.
\end{proof}
\subsection{Pohozaev identity for \texorpdfstring{\( \alpha \)}{α}-YMH fields}
In this section, we first derive a variational formula with respect to domain diffeomorphisms, which then serves as the foundation for establishing the Pohozaev identity in \cref{prop:Pohozaev}.
\begin{lem}
  Suppose \( \mathcal{F} \) is an associated \( G \)-bundle of a principal \( G \)-bundle \( \mathcal{P} \) over a Riemannian manifold \( (M,g) \), and let \( (A_\alpha, \phi_\alpha) \) be an \( \alpha \)-Yang--Mills--Higgs field on \( \mathcal{F} \). Then, for any vector field \( X \) compactly supported in \( M \), the following identity holds:
  \begin{equation}\label{eq:first-variational-formula-domain-diffeomorphism}
    \begin{split}
      &\int_M\left( \left( 1+\lvert \nabla _{A_\alpha}\phi_\alpha \rvert^2 \right)^{\alpha}+\lvert F_{A_\alpha} \rvert^2+\lvert \mu(\phi_\alpha) \rvert^2 \right)\div X dv_g\\
      &\qquad=2\int_M\left( 2\left\langle F_{A_\alpha}(\nabla _{e_i}X,e_j),F_{A_\alpha}(e_i,e_j) \right\rangle+f_\alpha\left\langle \nabla _{\nabla _{e_i}X}\phi_\alpha,\nabla _{e_i}\phi_\alpha \right\rangle\right)dv_g\\
      &\qquad\qquad+2\int_M\left\langle d\mu|_{\phi_\alpha}(\nabla _X\phi_\alpha),\mu(\phi_\alpha) \right\rangle dv_g.
    \end{split}
  \end{equation}
\end{lem}
\begin{proof}
  Let \( \varphi_t \) be a one-parameter family of diffeomorphisms of \( M \) generated by the vector field \( X \). Fix a smooth connection \( A_0 \) on \( \mathcal{F} \), and let \( \nabla_A \) denote the associated covariant derivative for a connection \( A \) (we will omit the subscript when the meaning is clear). For any pair \( (A,\phi) \), we define a family of pairs \( (\phi^t, A^t) := (\varphi_t^* \phi, \varphi_t^* A) \) as follows: Let \( \tau_t^0 \) denote the parallel transport in \( \mathcal{F} \) with respect to \( A_0 \) along the path \( \varphi_s(x) \), \( 0 \leq s \leq t \), for \( x \in M \). We define \( A^t := \varphi_t^* A \) via its covariant derivative
  \[
    \nabla _X^t v=(\tau_t^0)^{-1}\left( \nabla _{d\varphi_t(X)}\tau_t^0(v) \right),\quad \forall X\in \Gamma(TM),\, \forall v\in \Gamma(\mathcal{F}),
  \]
  and define \( \phi^t \) by
  \[
    \phi^t=\varphi_t^*\phi(x)=(\tau_t^0)^{-1}\left( \phi(\varphi_t(x)) \right).
  \]
  A direct computation shows that, when viewing \( F_{A^t} \) as a \( 2 \)-form on \( M \), we have
  \[
    F_{A^t}=(\tau_t^0)^{-1}\circ \varphi_t^*F_A.
  \]
  In fact, let \( v^t(x) = (\tau_t^0)^{-1} v(\varphi_t(x)) \) be a section at \( x \). Note that \( \nabla_Y^t v^t(x) = (\tau_t^0)^{-1} \left( \nabla_{d\varphi_t(Y)} v(\varphi_t(x)) \right) \) for any vector field \( Y \) at \( x \). By definition,
    \[
      F_{A^t}(Y, Z) v^t(x) = \nabla_Y^t \nabla_Z^t v^t(x) - \nabla_Z^t \nabla_Y^t v^t(x) - \nabla_{[Y, Z]}^t v^t(x).
    \]
    We compute each term:
    \[
      \nabla_Y^t \nabla_Z^t v^t(x) = (\tau_t^0)^{-1} \nabla_{d\varphi_t(Y)} \left( \nabla_{d\varphi_t(Z)} v(\varphi_t(x)) \right),
    \]
    and similarly for the other terms. Using the naturality of the bracket, we obtain
    \[
      F_{A^t}(Y, Z) v^t(x) = (\tau_t^0)^{-1}\left( F_A\big(d\varphi_t(Y), d\varphi_t(Z)\big) v(\varphi_t(x)) \right).
    \]
    Hence, \( F_{A^t} = (\tau_t^0)^{-1} \circ \varphi_t^* F_A \).

  Now, 
  \begin{align*}
    2\lVert F_{A^t} \rVert_{L^2}^2
    &=\int_M\left\langle F_{A^t},F_{A^t} \right\rangle dv_g
    =\int_M\left\langle (\tau_t^0)^{-1}\circ \varphi_t^*F_A,(\tau_t^0)^{-1}\circ \varphi_t^*F_A \right\rangle dv_g
    =\int_M\left\langle \varphi_t^*F_A,\varphi_t^*F_A \right\rangle dv_g,
  \end{align*}
  as parallel transport preserves the metric. By a change of coordinates,
  \begin{align*}
    2\lVert F_{A^t} \rVert_{L^2}^2
    &=\int_M\left\langle F_A|_{\varphi_t(x)}\left( d\varphi_t(e_i(x)),d\varphi_t(e_j(x)) \right),F_A|_{\varphi_t(x)}\left( d\varphi_t(e_i(x)),d\varphi_t(e_j(x)) \right) \right\rangle dv_g\\
    &=\int_{\varphi_t(M)\cong M}\left\langle F_A|_x\left( d\varphi_t(e_i(\varphi_t^{-1}(x))),d\varphi_t(e_j(\varphi_t^{-1}(x)))\right), \cdots \right\rangle J(\varphi_t^{-1})(x) dv_g,
  \end{align*}
  where \( dv_g(\varphi_t^{-1}(x)) = J(\varphi_t^{-1})(x) dv_g(x) \), and \( J(\cdot) \) denotes the Jacobian. Here, \( \{ e_i \}_{i=1}^2 \) is a local orthonormal frame of \( TM \), and we use the Einstein summation convention. Taking the derivative with respect to \( t \) at \( t = 0 \), and noting\footnote{Recall that \( L_X e_i := \lim_{t \to 0} \frac{d\varphi_{-t}(e_i(\varphi_t(x))) - e_i(x)}{t} = \left.\frac{d}{dt}\right|_{t=0} d\varphi_{-t}(e_i(\varphi_t(x))) \).} that \( \left.\frac{d}{dt}\right|_{t=0} d\varphi_t(e_i(\varphi_t^{-1}(x))) = -L_X e_i = -[X, e_i] \) and \( \left.\frac{d}{dt}\right|_{t=0} J(\varphi_t^{-1})(x) = -\operatorname{div} X \), we obtain
  \begin{align*}
    \left. \frac{d}{dt} \right\rvert_{t=0}\left\lVert F_{A^t} \right\rVert_{L^2}^2
      &=-2\int_M\left\langle F_A\left( [X,e_i],e_j \right),F_A\left( e_i,e_j \right) \right\rangle dv_g
      -\frac{1}{2}\int_M\lvert F_A \rvert^2\div Xdv_g\\
      &=-\frac{1}{2}\int_M\left( \lvert F_A \rvert^2\div X+4\left\langle F_A([X,e_i],e_j),F_A(e_i,e_j) \right\rangle \right)dv_g.
  \end{align*}
  Next, we compute the first variation of the energy related to \( \nabla ^t\phi^t \): 
  \begin{align*}
    \frac{1}{2}\int_M\lvert \nabla_{e_i} ^t\phi^t \rvert^2dv_g
              &=\frac{1}{2}\int_M\lvert (\tau_t^0)^{-1}\nabla _{d\varphi_t(e_i(x))}\phi(\varphi_t(x)) \rvert^2 dv_g
              =\frac{1}{2}\int_M\lvert \nabla _{d\varphi_t(e_i(x))}\phi(\varphi_t(x)) \rvert^2 dv_g\\
              &=\frac{1}{2}\int_{\varphi_t(M)\cong M}\lvert \nabla _{d\varphi_t(e_i(\phi_{t}^{-1}(x))}\phi(x) \rvert^2 J(\varphi_t^{-1}(x)) dv_g.
  \end{align*}
  Thus, the first variation is
  \[
    \left. \frac{d}{dt} \right\rvert_{t=0}\frac{1}{2}\int_{M}\lvert \nabla_{A^t}\phi^t \rvert^2=-\frac{1}{2}\int_M\lvert \nabla _{A}\phi \rvert^2\div Xdv_g+\int_M\left\langle \nabla _{-L_Xe_i}\phi,\nabla _{e_i}\phi \right\rangle dv_g.
  \]
  Also, it is clear that\footnote{In the classical case \( \mu(\phi) = \frac{\lambda}{4} (|\phi|^2 - m^2)^2 \), the function \( \mu(\phi) \) is invariant under the group action, which simplifies the variation.}
  \[
    \frac{1}{2}\int_M\lvert \mu(\phi^t) \rvert^2dv_g=\int_M\left\langle d\mu|_\phi(\nabla _X\phi),\mu(\phi) \right\rangle dv_g-\frac{1}{2}\int_M\lvert \mu(\phi) \rvert^2\div Xdv_g.
  \]
  Therefore, we finally obtain
  \begin{align*}
    &\left. \frac{d}{dt} \right\rvert_{t=0}\mathcal{L}_\alpha (A^t,\phi^t)\\
    &\qquad=-\frac{1}{2}\int_M\left( \left( 1+\lvert \nabla _A\phi \rvert^2 \right)^\alpha+\lvert F_A \rvert^2+\lvert \mu(\phi) \rvert^2 \right)\div Xdv_g\\
    &\qquad\qquad-\int_M\left( 2\left\langle F_A([X,e_i],e_j),F_A(e_i,e_j) \right\rangle+f_\alpha\left\langle \nabla _{[X,e_i]}\phi,\nabla _{e_i}\phi \right\rangle-\left\langle d\mu|_\phi(\nabla _X\phi),\mu(\phi) \right\rangle \right) dv_g.
  \end{align*}
  To further simplify, observe that since \( \nabla=\nabla _A \) is metric, we have
  \[
    \left\langle [X,e_i],e_k \right\rangle=\left\langle \nabla _Xe_i-\nabla _{e_i}X,e_k \right\rangle
    =-\left\langle \nabla _Xe_k,e_i \right\rangle-\left\langle \nabla _{e_i}X,e_k \right\rangle.
  \]
  Therefore,
  \begin{align*}
    \left\langle F_A([X,e_i],e_j),F_A(e_i,e_j) \right\rangle
    &=-\left( \left\langle \nabla _Xe_k,e_i \right\rangle+\left\langle \nabla _{e_i}X,e_k \right\rangle \right)\left\langle F_A(e_k,e_j),F_A(e_i,e_j) \right\rangle\\
    &=-\left\langle \nabla _{e_i}X,e_k \right\rangle\left\langle F_A(e_k,e_j),F_A(e_i,e_j) \right\rangle\\
    &=-\left\langle F_A(\nabla _{e_i}X,e_j),F_A(e_i,e_j) \right\rangle.
  \end{align*}
  Similarly,
  \[
    \left\langle \nabla _{[X,e_i]}\phi,\nabla _{e_i}\phi \right\rangle
    =-\left( \left\langle \nabla _Xe_k,e_i \right\rangle+\left\langle \nabla _{e_i}X,e_k \right\rangle \right)\left\langle \nabla _{e_k}\phi,\nabla _{e_i}\phi \right\rangle
    =-\left\langle \nabla _{\nabla _{e_i}X}\phi,\nabla _{e_i}\phi \right\rangle.
  \]
  Thus, the first variational formula for \( \alpha \)-YMH fields is
  \begin{align*}
    &\left. \frac{d}{dt} \right\rvert_{t=0}\frac{1}{2}\mathcal{L}_\alpha(A^t,\phi^t)\\
    &\qquad=-\frac{1}{2}\int_M\left( \left( 1+\lvert \nabla_A\phi \rvert^2 \right)^\alpha+\lvert F_A \rvert^2+\lvert \mu(\phi) \rvert^2 \right)\div Xdv_g\\
    &\qquad\qquad+\int_M\left( 2\left\langle F_A(\nabla _{e_i}X,e_j),F_A(e_i,e_j) \right\rangle+f_\alpha\left\langle \nabla _{\nabla _{e_i}X}\phi,\nabla _{e_i}\phi \right\rangle\footnotemark+\left\langle d\mu|_\phi(\nabla _X\phi),\mu(\phi) \right\rangle \right) dv_g.
  \end{align*}
  \footnotetext{Note that for maps between two manifolds we always have \( \nabla_{e_i} u = du(e_i) \) for the pullback connection of \( TM \otimes u^* TN \), which yields a vector field along the map \( u \).}
  In particular, if \( (A, \phi) = (A_\alpha, \phi_\alpha) \) is a critical point of \( \mathcal{L}_\alpha \), the above equality implies that the identity~\eqref{eq:first-variational-formula-domain-diffeomorphism} holds, thereby completing the proof.
\end{proof}

By choosing an appropriate variation vector field \( X \), we will obtain the Pohozaev identity.
\begin{proof}[Proof of \cref{prop:Pohozaev}]
  To prove \cref{prop:Pohozaev}, we regard \( \phi_\alpha \) as a map \( u_\alpha \mathpunct{:} D \to F \), and choose a vector field \( X \) on \( D \subset \Sigma = M \) as follows. For any \( 0 < t' < t \leq \rho < 1 \), let
  \[
    X = \eta(r)\, r\, \partial_r = \eta(|x|)\, x^i \partial_{x^i},
  \]
  where
  \[
    \eta(r)=
    \begin{cases}
      1,                              & r\leq t',\\
      \frac{t-r}{t-t'}, & t'\leq r\leq t, \\
      0,                              & r\geq t.
    \end{cases}
  \]
  For the orthonormal frame \( \{e_1 = \partial_r, e_2 = r^{-1} \partial_\theta\} \), it is easy to check:
  \begin{align*}
    \nabla _{\partial_r}\partial_r&=0,\\
    \nabla_{e_1}X&=(\eta r)'\partial_r=(\eta'r+\eta)\partial_r,\\
    \nabla _{e_2}X&=r^{-1}\nabla _{\partial\theta}(\eta(r) r\partial_r)
    =\eta \nabla _{\partial\theta}\partial_r=\eta \left( \Gamma_{\theta r}^r\partial_r+\Gamma_{\theta r}^\theta \partial_\theta \right)=\eta r^{-1}\partial_\theta,\\
    \div X&=\left\langle \nabla _{e_i}X,e_i \right\rangle=\eta'r+2\eta,
  \end{align*}
  as \( \Gamma_{\theta r}^r = 0 \) and \( \Gamma_{\theta r}^\theta = r^{-1} \). We also compute
  \begin{align*}
    f_\alpha\left\langle \nabla _{\nabla _{e_iX}}u_\alpha,\nabla _{e_i}u_\alpha \right\rangle
    &=f_\alpha\left( \eta'r\lvert \nabla _{A_\alpha;\partial_r}u_\alpha \rvert^2+\eta\lvert \nabla_{A_\alpha} u_\alpha \rvert^2 \right),\\
    \left\langle F_{A_\alpha}(\nabla _{e_i}X,e_j),F_{A_\alpha}(e_i,e_j) \right\rangle
    &=\eta' r^{-1}\lvert F_{r\theta} \rvert^2 +\eta \lvert F_{A_\alpha} \rvert^2=\left( \eta' r+\eta \right)\lvert F_{A_\alpha} \rvert^2.
  \end{align*}
  Thus, \eqref{eq:first-variational-formula-domain-diffeomorphism} reduces to
  \begin{equation}\label{eq:first-variational-formula-2-dim}
    \begin{split}
      &\int_D\left( \left( 1+\lvert \nabla _{A_\alpha}u_\alpha \rvert^2 \right)^{\alpha}+\lvert F_{A_\alpha} \rvert^2+\lvert \mu(u_\alpha) \rvert^2 \right)(\eta' r+2\eta)dv_g\\
      &\qquad =2\int_D\left( 2\left( \eta'r+\eta \right)\lvert F_{A_\alpha} \rvert^2+f_\alpha\left( \eta'r \lvert \nabla _{A_\alpha;\partial_ r}u_\alpha \rvert^2+\eta\lvert \nabla_{A_\alpha} u_\alpha \rvert^2 \right) \right)dv_g\\
      &\qquad\qquad+2\int_D\eta r\left\langle d\mu|_{u_\alpha}( \nabla _{A_{\alpha};\partial_r} u_\alpha),\mu(u_\alpha) \right\rangle dv_g.
    \end{split}
  \end{equation}
  That is,
  \begin{align*}
    &0=\frac{\alpha-1}{\alpha}\int_{D_{t}}\eta f_\alpha\lvert \nabla _{A_\alpha}u_\alpha \rvert^2dv_g\\
    &\qquad-\int_{D_{t}}\eta\left( \frac{1}{\alpha}f_\alpha-\lvert F_{A_\alpha} \rvert^2+\lvert \mu(u_\alpha) \rvert^2- r\left\langle d\mu|_{u_\alpha}(\nabla _{A_\alpha;\partial_r}u_\alpha),\mu(u_\alpha) \right\rangle\right)dv_g\\
    &\qquad+\frac{1}{t-t'}\int_{D_{t}\setminus D_{t'}}\frac{r}{2\alpha}f_\alpha\left( \lvert \nabla _{A_\alpha}u_\alpha \rvert^2-2\alpha \lvert \nabla _{A_\alpha;\partial_r}u_\alpha  \rvert^2 \right)dv_g\\
    &\qquad+\frac{1}{t-t'}\int_{D_{t}\setminus D_{t'}}\frac{r}{2}\left( \frac{1}{\alpha}f_\alpha-3\lvert F_{A_\alpha} \rvert^2+\lvert \mu(u_\alpha) \rvert^2 \right) dv_g.
  \end{align*}

  Letting \( t' \to t \) and invoking the estimates \( \|f_\alpha\|_{L^\infty(D)} \leq \Lambda' \) and \( \|F_{A_\alpha}\|_{L^\infty(D)} \leq C \), the proofs of which will be given in the next section (see \cref{lem:f-alpha-C0} and \cref{rmk:a_nabla_a_L_infinity} respectively), we obtain
  \begin{align*}
    &\int_{\partial D_{t}}f_\alpha\lvert \nabla _{A_\alpha;\partial_r}u_\alpha \rvert^2-\frac{1}{2\alpha}f_\alpha\lvert \nabla _{A_\alpha}u_\alpha \rvert^2
    +\int_{\partial D_{t}}\frac{3}{2}\lvert F_{A_\alpha} \rvert^2-\frac{1}{2}\int_{\partial D_{t}}\lvert \mu(u_\alpha) \rvert^2\\
    &\qquad=\frac{\alpha-1}{\alpha t}\int_{D_{t}}f_\alpha\lvert \nabla _{A_\alpha}u_\alpha \rvert^2+\int_{D_{t}}\frac{1}{t}\lvert F_{A_\alpha} \rvert^2+\int_{D_{t}}\left\langle d\mu|_{u_\alpha}(\nabla _{A_\alpha;\partial_r}u_\alpha),\mu(u_\alpha) \right\rangle+O(t),
  \end{align*}
  that is 
  \[
    \int_{\partial D_{t}}f_\alpha\lvert \nabla _{A_\alpha;\partial_r}u_\alpha \rvert^2
    -\frac{1}{2\alpha}f_\alpha\lvert \nabla _{A_\alpha}u_\alpha \rvert^2
    =\frac{\alpha-1}{\alpha t}\int_{D_{t}}f_\alpha\lvert \nabla _{A_\alpha}u_\alpha \rvert^2
    +O(t).
  \]
\end{proof}
\section{The energy identity}\label{sec:EI}
In this section, we will establish the energy identity stated in \cref{thm:main}. We begin by reducing the full \( \mathcal{L}_\alpha \)-energy identity to a formulation involving only the \( \alpha \)-energy of \( \nabla_{A_\alpha} u_\alpha \) on the neck region; see \cref{prop:equivalence-of-EI}. Next, we apply the Hodge decomposition and utilize the new conservation law to reveal the hidden Jacobian structure of \( \alpha \)-YMH fields. This structure allows us to derive first a weaker version of the \( \alpha \)-energy identity for the localized energy \( E(A_\alpha, u_\alpha; D) \); see \cref{lem:EI-E-A-u-D}. We then employ the Pohozaev identity to improve the estimate on \( \lVert f_\alpha \rVert_{C^0(D)} \), showing that it tends to \( 1 \) as \( \alpha \to 1 \); see \cref{lem:f_alpha-C0-equal-1}. This ultimately yields the desired \( \alpha \)-energy identity. The proof of the no neck property will be presented separately in \cref{sec:NN}.

By a reduction argument as in \cites{DingTian1995Energy, LiWang2010Weak}, we may assume without loss of generality that there is only one bubble. Under this assumption, we obtain the following small energy property on each neck region \( D_\delta \setminus D_{r_\alpha R} \): for all \( t \in [r_\alpha R, \delta] \),
\begin{equation}\label{eq:small-energy-on-each-annular-neck}
  \begin{split}
    \|\nabla_{A_\lambda} u_\lambda\|_{L^2(D_2 \setminus D_1)}
    &= \|\nabla_A u\|_{L^2(D_{2t} \setminus D_t)} < \epsilon \leq \epsilon_0,\\
    \|F_{A_\lambda}\|_{L^2(D_2 \setminus D_1)}
    &= \lambda \|F_A\|_{L^2(D_{2t} \setminus D_t)} < \delta_0,
  \end{split}
\end{equation}
where \( A_\lambda(x) = \lambda a(\lambda x) \) and \( u_\lambda(x) = u(\lambda x) \) with \( \lambda = r_\alpha \) as before. Note that the second relation between \( F_{A_\lambda} \) and \( F_A \) always holds as \( r_\alpha \to 0 \) when \( \alpha \to 1 \).

We now decompose the disc \( D \) into three regions: the strong convergence region \( D \setminus D_\delta \); the neck region \( D_\delta \setminus D_{r_\alpha R} \); and the bubbling region \( D_{r_\alpha R} \). For simplicity, we do not consider here the neck region between two bubbles, but restrict ourselves to the region of the limit map and a single bubble. In the case of multiple bubbles, one needs to consider the rescaled energy in the neck region between bubbles; however, the arguments are analogous, with the energy and small energy estimates replaced by their rescaled counterparts.

On the neck region \( T_{\delta, r_\alpha R} \mathpunct{:}= D_\delta \setminus D_{r_\alpha R} \), by \eqref{eq:small-energy-on-each-annular-neck}, we have the following estimate: for any \( x \in D_{3/2} \setminus D_{4/3} \), the small energy estimate (see \cref{prop:small-energy-estimates-scaled}) yields
\[
  |a_\lambda(x)| + |\nabla u_\lambda(x)| \leq C \left( \|\nabla_{A_\lambda} u_\lambda\|_{L^2(D_2 \setminus D_1)} + \|F_{A_\lambda}\|_{L^2(D_2 \setminus D_1)} \right).
\]
Scaling back to the original variables, we obtain
\[
  \lambda |a(\lambda x)| + \lambda |\nabla u(\lambda x)| \leq C \left( \|\nabla_A u\|_{L^2(D_{2t} \setminus D_t)} + \lambda \|F_A\|_{L^2(D_{2t} \setminus D_t)} \right).
\]
Therefore, for any \( x \in D_{3t/2} \setminus D_{4t/3} \),
\begin{equation}\label{eq:nabla_Au_in_L_2_infnity}
|\nabla_{A_\alpha} u_\alpha(x)| \leq C \left( |a_\alpha(x)| + |\nabla u_\alpha(x)| \right) \leq C \left( \frac{\epsilon}{|x|} + \|F_{A_\alpha}\|_{L^2(D_{2t} \setminus D_t)} \right) \leq \frac{C\epsilon}{|x|},
\end{equation}
provided that \( \delta \) is sufficiently small (here we apply the elliptic estimates for the scaled equation of \( A \) with scaling factor \( \delta \)). This further implies
\begin{equation}\label{eq:L-2-infty-u}
  \lVert \nabla_{A_\alpha} u_\alpha \rVert_{L^{2,\infty}(T_{\delta,r_\alpha R})}\leq C\epsilon.
\end{equation}

In fact, we can extend the estimate \eqref{eq:nabla_Au_in_L_2_infnity} to the whole disc \( D \). 

\begin{clm}
If \( \alpha-1 \) is sufficiently small, then
\begin{equation}\label{eq:nabla-u-alpha-c0}
  \lVert \nabla_{A_\alpha}  u_\alpha \rVert_{C^0(D)}\leq Cr_\alpha^{-1}.
\end{equation}
\end{clm}
\begin{proof}
  Clearly, it suffices to prove \eqref{eq:nabla-u-alpha-c0} on \( D_{r_\alpha R} \) and \( D \setminus D_\delta \). 
  First, observe that there exists a constant \( \delta' > 0 \) (independent of \( \alpha \)) such that, for \( \lambda=r_\alpha \),
  \[
    \lVert \nabla_{A_\lambda} u_\lambda \rVert_{L^2(D(y,2\delta'))} < \epsilon_0
    \quad \text{and} \quad
    \lVert F_{A_\lambda} \rVert_{L^2(D(y,2\delta'))} < \delta_0
  \]
  for any \( y \in D_R \). This follows since, for any \( R > 0 \) and \( y \in D_R \), we have \( u_\lambda(x) = u_\alpha(r_\alpha x) \to \omega(x) \) in \( C^\infty(D_R) \) and \( a_\lambda(x) = r_\alpha a_\alpha(r_\alpha x) \to 0 \) in \( C^\infty(D_R) \) (because we may assume that \( A_\lambda \) is in Coulomb gauge, so \( \lVert a_\lambda \rVert_{W^{1,2}(D_R)} \leq C \lVert F_{A_\lambda} \rVert_{L^2(D_R)} = C r_\alpha \lVert F_A \rVert_{L^2(D_{r_\alpha R})} \to 0 \)). Applying \cref{prop:small-energy-estimates-scaled}, we see that for \( \alpha - 1 \) sufficiently small,
  \begin{equation}\label{eq:nabla-A-u-lambda-C0-bounded}
    \lvert \nabla _{A_\lambda}u_\lambda \rvert(y)\leq |a_\lambda(y)| + |\nabla u_\lambda(y)| \leq \frac{C}{\delta'}(\epsilon_0 + \delta_0) \leq C.
  \end{equation}
  Scaling back, we obtain that for any \( x \in D_{r_\alpha R} \),
  \[
    |a_\alpha(x)| + |\nabla u_\alpha(x)| = r_\alpha^{-1}\left( |a_\lambda(r_\alpha^{-1}x)| + |\nabla u_\lambda(r_\alpha^{-1} x)| \right) \leq C r_\alpha^{-1}.
  \]

  Similarly, since \( u_\alpha \to u_\infty \) and \( a_\alpha \to a_\infty \) smoothly on \( D \setminus D_\delta \), it follows from \cref{prop:small-energy-estimate} that, for \( \alpha - 1 \) sufficiently small,
  \[
    |a_\alpha| + |\nabla u_\alpha| \leq C \delta^{-1}.
  \]
  Combining these two cases, we conclude that the claim holds.
\end{proof}
\begin{lem}\label{lem:f-alpha-C0}
  There exists a constant \( \Lambda' > 0 \), depending only on \( \Lambda \) and \( \epsilon_0 \), such that when \( \alpha-1 \) is sufficiently small,
  \begin{equation}\label{eq:f-alpha-C0}
    \lVert f_\alpha \rVert_{C^0(D)} \leq \Lambda'.
  \end{equation}
\end{lem}
\begin{proof}
  Note that for any \( R > 0 \), we have
  \[
    \int_{D_R}\lvert \nabla_{ A_\lambda}  u_\lambda(x) \rvert^{2\alpha}
    =r_\alpha^{2\alpha-2}\int_{D_{r_\alpha R}}\lvert \nabla_{A_\alpha}  u_\alpha(x) \rvert^{2\alpha}
    \leq Cr_\alpha^{2\alpha-2}\int_{D_{r_\alpha R}} \left( \left( 1+\lvert \nabla _{A_\alpha}u_\alpha \rvert^2 \right)^{\alpha} -1 \right)
    \leq C r_\alpha^{2\alpha-2} \Lambda.
  \]
  On the other hand, \eqref{eq:nabla-A-u-lambda-C0-bounded} implies that
  \[
    \lim_{\alpha \to 1} \lVert \nabla_{A_\lambda} u_\lambda \rVert_{C^0(\mathbb{R}^2)} < +\infty,
  \]
  and
  \[
  \lim_{R\to\infty}\lim_{\alpha\to1}\int_{D_R}\lvert \nabla_{ A_\lambda}   u_\lambda(x) \rvert^{2\alpha}dx=\lim_{R\to\infty}\int_{D_R}\lvert \nabla \omega \rvert^2\geq 2\epsilon_0^2.
  \]
  In conclusion, we obtain
  \[
    \epsilon_0^2 \leq C \Lambda \liminf_{\alpha \to 1} r_\alpha^{2\alpha-2},
  \]
  that is, for \( 1 < \alpha < \alpha_0 \),
  \begin{equation}\label{eq:r-alpha-2-2alpha-unifom-bounded}
    r_\alpha^{2-2\alpha} \leq \frac{C \Lambda}{\epsilon_0^2}.
  \end{equation}
  Combining this with \eqref{eq:nabla-u-alpha-c0}, we deduce
  \[
    \lVert f_\alpha \rVert_{C^0(D)} \leq \alpha (1 + \lVert \nabla_{A_\alpha} u_\alpha \rVert_{C^0(D)}^2)^{\alpha-1}
    \leq C r_\alpha^{2-2\alpha} \leq \frac{C \Lambda}{\epsilon_0^2} =: \Lambda'.
  \]
\end{proof}

Now, we turn to the study of the energy identity. On the strong convergence region \( D \setminus D_\delta \), since each point admits a small neighborhood where the small energy condition is satisfied and there is no energy concentration, the small energy estimates (\cref{prop:small-energy-estimate}) imply that \( u_\alpha \to u_\infty \) in \( C^\infty(D \setminus D_\delta) \) and \( A_\alpha \to A_\infty \) in \( C^\infty(D \setminus D_\delta) \). Consequently,
\begin{equation}\label{eq:EI-D-D-delta}
  \lim_{\alpha\to1} \mathcal{L}_\alpha(A_\alpha, u_\alpha; D \setminus D_\delta)
  =\frac{1}{2}\int_{D\setminus D_\delta}\lvert \nabla _{A_\infty}u_\infty \rvert^2 +\lvert F_A \rvert^2+\lvert \mu(u) \rvert^2
  = \mathcal{L}(A_\infty, u_\infty; D \setminus D_\delta).
\end{equation} 

To analyze the energy convergence in the bubbling region \( D_{r_\alpha R} \), we first note the different behavior of the map \( u \) and the connection 1-form \( a \). Unlike the map \( u \), which satisfies the small energy condition only on annular neck regions of fixed ratios, the connection 1-form \( a \) scales in a favorable manner, allowing the small energy condition to hold over the entire small disc. Specifically, choose \( \lambda = \delta_1 \) for some \( \delta_1 > 0 \) depending only on an upper bound for \( \lVert F_{A_\alpha} \rVert_{L^2(\Sigma)} \). If \( \delta_1 \) is sufficiently small, we have \( \lVert F_{A_\lambda} \rVert_{L^2(D_2)} = \lambda \lVert F_A \rVert_{L^2(D_{2\delta_1})} < \delta_0 \). Then, by the small energy estimates (\cref{prop:small-energy-estimates-scaled}),
\[
  \lVert a_\lambda \rVert_{L^\infty(D_1)} + \lVert \nabla a_\lambda \rVert_{L^8(D_1)}
  \leq C \lVert a_\lambda \rVert_{W^{2,8/5}(D_1)}
  \leq C \left( \lVert \nabla_{A_\lambda} u_\lambda \rVert_{L^2(D_2)} + \lVert F_{A_\lambda} \rVert_{L^2(D_2)} \right).
\]
Scaling back, for any \( \delta < \delta_1 \),
\[
  \lVert a \rVert_{L^\infty(D_\delta)}
  \leq \lVert a \rVert_{L^\infty(D_{\delta_1})}
  \leq \frac{C}{\delta_1} \left( \lVert \nabla_A u \rVert_{L^2(D_{2\delta_1})} + \lVert F_A \rVert_{L^2(D_{2\delta_1})} \right),
\]
and
\[
  \lVert \nabla a \rVert_{L^8(D_\delta)}
  \leq \frac{C}{\delta_1^{7/4}} \left( \lVert \nabla_A u \rVert_{L^2(D_{2\delta_1})} + \lVert F_A \rVert_{L^2(D_{2\delta_1})} \right).
\]
In particular, this implies
\[
  \lVert F_A \rVert_{L^8(D_\delta)}
  \leq C \lVert \lvert\nabla a\rvert + \lvert a\rvert^2 \rVert_{L^8(D_\delta)}
  \leq \frac{C}{\delta_1^{7/4}} \left( \lVert \nabla_A u \rVert_{L^2(D_{2\delta_1})} + \lVert F_A \rVert_{L^2(D_{2\delta_1})} \right).
\]
Therefore, by H\"older's inequality,
\begin{equation}\label{eq:convergence-of-rescaled-FA}
  r_\alpha^{-2} \int_{D_R} \lvert F_{A_\lambda}\rvert^2 = \int_{D_{r_\alpha R}} \lvert F_A\rvert^2 \leq C \lVert F_A \rVert_{L^8(D_\delta)}^2 r_\alpha R\sqrt{r_\alpha R}.
\end{equation}

\begin{rmk}\label{rmk:a_nabla_a_L_infinity}
  With the help of \( \lVert f_\alpha \rVert_{C^0(D)} \leq \Lambda' \), we can estimate \( \lVert \nabla a \rVert_{L^\infty(D)} \). By the existence of local Coulomb gauge (see \cite{Uhlenbeck1982Connections}*{Thm.~1.3}), up to a choice of gauge, we have
  \[
    \lVert a \rVert_{W^{1,p}(D_{2\delta_1})} \leq C \lVert F_A \rVert_{L^2(D_{2\delta_1})}, \quad \forall\, p \geq 1.
  \]
  In particular,
  \[
    \lVert \nabla a \rVert_{L^{2\alpha}(D_{2\delta_1})} \leq C \lVert F_A \rVert_{L^2(D_{2\delta_1})}.
  \]
  Then, the \( L^p \) estimates applied to the equation for \( a \) (see \eqref{eq:alpha-YMH-local}) yield
  \[
    \lVert a \rVert_{W^{2,2\alpha}(D_{\delta_1})} \leq C \left( \lVert f_\alpha \rVert_{L^\infty(D_{2\delta_1})} \lVert \nabla_A u \rVert_{L^{2\alpha}(D_{2\delta_1})} + \lVert \nabla a \rVert_{L^{2\alpha}(D_{2\delta_1})} \lVert a \rVert_{L^\infty(D_{2\delta_1})} + \lVert a \rVert_{L^\infty(D_{2\delta_1})}^3 \right),
  \]
  so that Sobolev's embedding theorem implies
  \[
    \lVert \nabla a \rVert_{L^\infty(D_{\delta_1})} \leq C(\Lambda', \Lambda, \delta_1)=C(\epsilon_0,\Lambda,\delta_1).
  \]
  By a covering argument,
  \[
    \lVert a \rVert_{L^\infty(D)} + \lVert \nabla a \rVert_{L^\infty(D)} + \lVert \nabla^2 a \rVert_{L^{2\alpha}(D)} \leq C(\epsilon_0,\Lambda,\delta_1).
  \]
\end{rmk}

Now, on the bubbling region \( D_{r_\alpha R} \), recall that \( u_\lambda \to \omega \) in \( C^\infty(D_R) \) and \( a_\lambda \to 0 \) in \( C^\infty(D_R) \), and note that \( \mathcal{L}_\alpha = \mathcal{L}_{\alpha,1} \). By \eqref{eq:scaled-energy-relation}, for \( \lambda = r_\alpha \), we have
\begin{align*}
  &\lim_{\alpha \to 1} \mathcal{L}_\alpha(A_\alpha, u_\alpha; D_{r_\alpha R}) \\
      &\qquad=\lim_{\alpha\to1}\frac{1}{2}r_\alpha^{2-2\alpha}\int_{D_R}\left( r_\alpha^2+\lvert \nabla _{A_{\lambda}}u_\lambda \rvert \right)^\alpha+r_\alpha^{2\alpha-4}\lvert F_{A_\lambda} \rvert^2+r_\alpha^{2\alpha}\left( \lvert \mu(u_\lambda) \rvert^2-1 \right) \\
      &\qquad=\lim_{\alpha\to1}\frac{1}{2}r_\alpha^{2-2\alpha}\int_{D_R}\lvert d\omega \rvert^2+\lim_{\alpha\to1}r_\alpha^{-2}\int_{D_R}\lvert F_{A_\lambda} \rvert^2
      =\bar\mu E(\omega;D_R),
\end{align*}
where we have used \eqref{eq:convergence-of-rescaled-FA} and the fact that \( r_\alpha^{2-2\alpha} \geq 1 \) is bounded above, see \eqref{eq:r-alpha-2-2alpha-unifom-bounded}; here, we take a subsequence of \( (A_\alpha, u_\alpha) \) such that \( \lim_{\alpha \to 1} r_\alpha^{2-2\alpha} = \bar\mu \). Therefore,
\begin{equation}\label{eq:EI-D-r-alpha-R}
  \lim_{\alpha \to 1} \mathcal{L}_\alpha(A_\alpha, u_\alpha; D_{r_\alpha R}) = \bar\mu E(\omega; D_R).
\end{equation}

Thus, by first letting \( R \to +\infty \) and then \( \delta \to 0 \) in \eqref{eq:EI-D-D-delta} and \eqref{eq:EI-D-r-alpha-R}, we obtain:
\begin{align*}
  &\lim_{\alpha \to 1} \mathcal{L}_\alpha(A_\alpha, u_\alpha; D) \\
  &\quad = \lim_{\delta \to 0} \lim_{R \to +\infty} \lim_{\alpha \to 1} \mathcal{L}_\alpha\left(A_\alpha, u_\alpha; \left(D \setminus D_\delta\right) \cup \left(T_{\delta,r_\alpha R}\right) \cup D_{r_\alpha R}\right) \\
  &\quad = \mathcal{L}(A_\infty, u_\infty; D) + \bar\mu E(\omega)
      + \lim_{\delta \to 0} \lim_{R \to +\infty} \lim_{\alpha \to 1} \mathcal{L}_\alpha(A_\alpha, u_\alpha; T_{\delta,r_\alpha R}).
\end{align*}
Moreover, note that
\[
  \lim_{\delta\to0}\lim_{\alpha\to1}\frac{1}{2}\int_{D_\delta}\left( \lvert F_{A_\alpha} \rvert^2+\lvert \mu(u) \rvert^2-1 \right) dv_g=0,
\]
which can be proved by \cref{rmk:a_nabla_a_L_infinity}. We summarize the above convergence of energy on \( D \) in the following form.
\begin{prop}\label{prop:equivalence-of-EI}
  Under the assumptions of \cref{thm:main}, the energy identity \eqref{eq:alpha-EI} in \cref{thm:main} for the \( \alpha \)-energy of \( \alpha \)-YMH fields is equivalent to the condition \( \bar\mu=1 \) and
  \begin{equation}\label{eq:equivalence-of-EI}
    \lim_{\delta \to 0} \lim_{R \to \infty} \lim_{\alpha \to 1} \frac{1}{2} \int_{T_{\delta,r_\alpha R}} \left( 1 + \lvert\nabla_{A_\alpha} u_\alpha\rvert^2 \right)^\alpha \, dv_g = 0.
  \end{equation}
\end{prop}

To prove the above proposition, we first establish the following weaker version of the energy identity, which asserts that the \( E(A_\alpha, u_\alpha; D) \)-energy identity holds, rather than the full \( \mathcal{L}_\alpha \)-energy.

\begin{lem}\label{lem:EI-E-A-u-D}
  Under the assumptions of \cref{thm:main}, for any \( \epsilon > 0 \), there exist \( \alpha_0 > 1 \), \( R_0 > 0 \), and \( \delta_1 > 0 \) such that for all \( 1 < \alpha \leq \alpha_0 \), \( R > R_0 \), and \( \delta < \delta_1 \), we have
  \[
    E(A_\alpha, u_\alpha; T_{\delta,r_\alpha R}) = \frac{1}{2} \int_{T_{\delta,r_\alpha R}} \lvert\nabla_{A_\alpha} u_\alpha\rvert^2 \leq C \epsilon^2,
  \]
  that is,
  \[
    \lim_{\delta \to 0} \lim_{R \to +\infty} \lim_{\alpha \to 1} E(A_\alpha, u_\alpha; T_{\delta,r_\alpha R}) = 0.
  \]
\end{lem}

\begin{proof}
  The proof is rather lengthy, so we divide it into several steps.

  \step[The Hodge decomposition] We begin by constructing a function \( \tilde u_\alpha \) supported in \( B_{2\delta} \) such that 
  \[
    \tilde u_\alpha|_{D_{r_\alpha R}} = \bar{u}_\alpha^2 - \bar{u}_\alpha^1,
  \]
  where
  \[
    \bar{u}_\alpha^1 = \frac{1}{\lvert D_{2\delta} \setminus D_{\delta} \rvert} \int_{D_{2\delta} \setminus D_{\delta}} u_\alpha(x) \, dx, \quad
    \bar{u}_\alpha^2 = \frac{1}{\lvert D_{2r_\alpha R} \setminus D_{r_\alpha R} \rvert} \int_{D_{2r_\alpha R} \setminus D_{r_\alpha R}} u_\alpha(x) \, dx.
  \]
  In fact, it suffices to define
  \[
    \tilde u_\alpha(x) = \varphi_\delta(x)\left[ \left(1 - \varphi_{r_\alpha R}(x)\right)\left(u_\alpha(x) - \bar{u}_\alpha^2 \right) + \bar{u}_\alpha^2 - \bar{u}_\alpha^1 \right],
  \]
  where \( \varphi \in C_0^\infty(D_2) \) satisfies \( \varphi|_{D_1} \equiv 1 \), and we denote \( \varphi_t(x) := \varphi(x/t) \).

  Consider now the Hodge decomposition (see \cite{Helein2002Harmonic}*{Prop.~3.3.9}):
  \begin{equation}\label{eq:Hodge-decomposition}
    f_\alpha \nabla_{A_\alpha} \tilde u_\alpha = \nabla D_\alpha + \nabla^\perp Q_\alpha,
  \end{equation}
  where \( D_\alpha, Q_\alpha \in W^{1,2}_0(D) \).

  To estimate \( \lVert \nabla Q_\alpha \rVert_{L^2(D)} \), we use \eqref{eq:Hodge-decomposition} to compute:
  \[
    \Delta Q_\alpha = \curl \nabla^\perp Q_\alpha = \curl \left( f_\alpha \nabla_{A_\alpha} \tilde u_\alpha \right)
    = -\nabla(f_\alpha - 1) \cdot \nabla_{A_\alpha}^\perp \tilde u_\alpha
    = -\div \left( (f_\alpha - 1) \nabla_{A_\alpha}^\perp \tilde u_\alpha \right).
  \]
  Applying \cref{lem:div-elliptic-estimate} together with \eqref{eq:f-alpha-C0}, we obtain
  \[
    \lVert \nabla Q_\alpha \rVert_{L^2} \leq \lVert (f_\alpha - 1) \nabla_{A_\alpha}^\perp \tilde u_\alpha \rVert_{L^2}
    = \left\lVert \frac{f_\alpha - 1}{f_\alpha} \cdot f_\alpha \nabla_{A_\alpha}^\perp \tilde u_\alpha \right\rVert_{L^2}
    \leq \left(1 - \frac{1}{\Lambda'}\right) \lVert f_\alpha \nabla_{A_\alpha} \tilde u_\alpha \rVert_{L^2}.
  \]
  On the other hand, from \eqref{eq:Hodge-decomposition}, we also have
  \[
    \lVert f_\alpha \nabla_{A_\alpha} \tilde u_\alpha \rVert_{L^2(D)} \leq \lVert \nabla D_\alpha \rVert_{L^2(D)} + \lVert \nabla Q_\alpha \rVert_{L^2(D)},
  \]
  which leads to the estimate
  \begin{equation}\label{eq:f-nabla-u-by-nabla-D}
    \lVert f_\alpha \nabla_{A_\alpha} \tilde u_\alpha \rVert_{L^2(D)} \leq \Lambda' \lVert \nabla D_\alpha \rVert_{L^2(D)}.
  \end{equation}

\step[The hidden Jacobian structure] To proceed, we estimate \( \lVert \nabla D_\alpha \rVert_{L^2(D)} \). Here, the spherical structure of the fiber plays a crucial role. The method is inspired by H\'elein's approach to the regularity theory for harmonic maps into spheres (see \cite{Helein2002Harmonic}).

Observe that the covariant derivative of \( \tilde u_\alpha \) naturally decomposes into three components:
\[
  \nabla_{A_\alpha} \tilde u_\alpha(x) = \nabla \varphi_\delta(x)(u_\alpha(x) - \bar{u}_\alpha^1) - \nabla \varphi_{r_\alpha R}(x)(u_\alpha(x) - \bar{u}_\alpha^2) + \varphi_\delta(x)(1 - \varphi_{r_\alpha R}(x)) \nabla_{A_\alpha} u_\alpha(x).
\]
Therefore, we compute:
\begin{equation}\label{eq:div-f-nabla-u}
  \begin{multlined}
    \div(f_\alpha \nabla_{A_\alpha} \tilde u_\alpha) = \\
    \div\bigl(f_\alpha \varphi_\delta(1 - \varphi_{r_\alpha R}) \nabla_{A_\alpha} u_\alpha\bigr)
    + \div\bigl(f_\alpha \nabla \varphi_\delta (u_\alpha - \bar{u}_\alpha^1)\bigr)
    - \div\bigl(f_\alpha \nabla \varphi_{r_\alpha R}(u_\alpha - \bar{u}_\alpha^2)\bigr).
  \end{multlined}
\end{equation}

We now estimate the first term. When the fiber is the standard sphere, then \( \lvert u_\alpha \rvert \equiv 1 \), so that
\[
  \sum_{j=1}^K u_\alpha^j \nabla_{A_\alpha} u_\alpha^j = 0 = \sum_{j=1}^K u_\alpha^j \nabla u_\alpha^j,
\]
where we abuse the nation by using \( \nabla_{A_\alpha} u_\alpha^j \) to denotes the \( j \)-th component of the covariant derivative \( \nabla_{A_\alpha} u_\alpha \). Moreover, by the Euler--Lagrange equation \eqref{eq:alpha-YMH-local}, we have
\begin{equation}\label{eq:rewrite-EL}
  \begin{split}
    \div(f_\alpha \nabla_{A_\alpha} u_\alpha^i)
    &= -f_\alpha \sum_{j=1}^K u_\alpha^i \lvert \nabla_{A_\alpha} u_\alpha^j \rvert^2 - \Phi_\alpha^i(A_\alpha, u_\alpha) \\
    &= \sum_{j=1}^K f_\alpha \left( \nabla_{A_\alpha} u_\alpha^i u_\alpha^j - \nabla_{A_\alpha} u_\alpha^j u_\alpha^i \right) \nabla_{A_\alpha} u_\alpha^j - \Phi_\alpha^i(A_\alpha, u_\alpha),
  \end{split}
\end{equation}
where \( \Phi_\alpha(A_\alpha, u_\alpha) = \bigl(\Phi_\alpha^1(A_\alpha, u_\alpha), \ldots, \Phi_\alpha^K(A_\alpha, u_\alpha)\bigr) \).
Applying \eqref{eq:rewrite-EL}, we compute:
\begin{align*}
  &\div\left( f_\alpha \varphi_\delta(1 - \varphi_{r_\alpha R}) \nabla_{A_\alpha} u_\alpha^i \right) \\
  &\quad = \varphi_\delta(1 - \varphi_{r_\alpha R}) \div(f_\alpha \nabla_{A_\alpha} u_\alpha^i)
    + f_\alpha \nabla_{A_\alpha} u_\alpha^i \cdot \nabla(\varphi_\delta(1 - \varphi_{r_\alpha R})) \\
  &\quad = \varphi_\delta(1 - \varphi_{r_\alpha R}) \sum_{j=1}^K f_\alpha \left( \nabla_{A_\alpha} u_\alpha^i u_\alpha^j - \nabla_{A_\alpha} u_\alpha^j u_\alpha^i \right) \nabla_{A_\alpha} u_\alpha^j \\
  &\qquad + f_\alpha \nabla_{A_\alpha} u_\alpha^i \cdot \nabla(\varphi_\delta(1 - \varphi_{r_\alpha R})) - \varphi_\delta(1 - \varphi_{r_\alpha R}) \Phi_\alpha^i(A_\alpha, u_\alpha).
\end{align*}
A direct computation reveals further cancellation:
\begin{align*}
  &\div\left( f_\alpha \varphi_\delta(1 - \varphi_{r_\alpha R}) \nabla_{A_\alpha} u_\alpha^i \right) \\
  &\quad = \sum_{j=1}^K f_\alpha \left( \nabla_{A_\alpha} u_\alpha^i u_\alpha^j - \nabla_{A_\alpha} u_\alpha^j u_\alpha^i \right) \nabla_{A_\alpha} \left( \varphi_\delta(1 - \varphi_{r_\alpha R}) u_\alpha^j \right) \\
  &\qquad - \sum_{j=1}^K f_\alpha \left( \nabla_{A_\alpha} u_\alpha^i u_\alpha^j - \nabla_{A_\alpha} u_\alpha^j u_\alpha^i \right) u_\alpha^j \nabla(\varphi_\delta(1 - \varphi_{r_\alpha R})) \\
  &\quad\qquad + f_\alpha \nabla_{A_\alpha} u_\alpha^i \cdot \nabla(\varphi_\delta(1 - \varphi_{r_\alpha R})) - \varphi_\delta(1 - \varphi_{r_\alpha R}) \Phi_\alpha^i(A_\alpha, u_\alpha) \\
  &\quad = \sum_{j=1}^K f_\alpha \left( \nabla_{A_\alpha} u_\alpha^i u_\alpha^j - \nabla_{A_\alpha} u_\alpha^j u_\alpha^i \right) \nabla_{A_\alpha} \left( \varphi_\delta(1 - \varphi_{r_\alpha R}) u_\alpha^j \right) \\
  &\qquad - \varphi_\delta(1 - \varphi_{r_\alpha R}) \Phi_\alpha^i(A_\alpha, u_\alpha).
\end{align*}

Moreover, by the conservation law (see \eqref{eq:conservation-law-component}),
\begin{align*}
  \div \left( f_\alpha \left( \nabla_{A_\alpha} u_\alpha^i u_\alpha^j - \nabla_{A_\alpha} u_\alpha^j u_\alpha^i \right) \right)
  = -2 \div \left( \tr \left[ F_A, a \right]^i_j \right)
  = 2 \partial_\beta \left[ F_{\gamma\beta}, a_\gamma \right]^i_j,
\end{align*}
which, unlike the cases of harmonic maps \cite{LiZhu2019Energy}*{(3.11)} or Dirac--harmonic maps \cite{LiLiuZhuZhu2021Energy}*{(3.3)}, where the term is generally zero, this term is not zero in the YMH theory, adding extra complexity.

By the \( L^2 \) Hodge decomposition (see e.g. \cite{Schwarz1995Hodge}*{Thm.~2.4.2}), we may write
\begin{equation}\label{eq:u-nabla-u-hodege-decompo}
  f_\alpha \left( \nabla_{A_\alpha} u_\alpha^i u_\alpha^j - \nabla_{A_\alpha} u_\alpha^j u_\alpha^i \right)
  = \nabla^\perp G_{\alpha,1}^{ij} + \nabla G_{\alpha,2}^{ij},
\end{equation}
with \( G_{\alpha,1}^{ij} \in W^{1,2}(D) \) and \( G_{\alpha,2}^{ij} \in W_0^{1,2}(D) \) satisfying
\begin{equation}\label{eq:G-alpha-2}
  \Delta G_{\alpha,2}^{ij} = 2 \partial_\beta [F_{\gamma\beta}, a_\gamma]^i_j = -2 \div \left( \tr [F_{A_\alpha}, a_\alpha]^i_j \right).
\end{equation}

In conclusion, combining \eqref{eq:Hodge-decomposition} with \eqref{eq:div-f-nabla-u}, we obtain:
\begin{align}\label{eq:Laplace-D}
    \Delta D_\alpha &= \div(f_\alpha \nabla_{A_\alpha} \tilde u_\alpha) \nonumber\\
    &= \nabla^\perp G_{\alpha,1} \cdot \nabla_{A_\alpha} \left( \varphi_\delta (1 - \varphi_{r_\alpha R}) u_\alpha \right)
    + \nabla G_{\alpha,2} \cdot \nabla_{A_\alpha} \left( \varphi_\delta (1 - \varphi_{r_\alpha R}) u_\alpha \right) \\
    &\qquad - \varphi_\delta (1 - \varphi_{r_\alpha R}) \Phi_\alpha(A_\alpha, u_\alpha)
    + \div \left( f_\alpha \left( \nabla \varphi_\delta (u_\alpha - \bar{u}_\alpha^1) - \nabla \varphi_{r_\alpha R} (u_\alpha - \bar{u}_\alpha^2) \right) \right).\nonumber
\end{align}
The first term in the RHS is the \emph{hidden Jocobian structure} of \( \alpha \)-YMH fields. 
\step[Estimates of \( \lVert D_\alpha \rVert_{L^2} \)] 
Next, we consider the following equations on \( D \) with vanishing boundary conditions:
\begin{align}
  \Delta\Phi_1 &= \nabla^\perp G_{\alpha,1} \nabla\left( \varphi_\delta(1 - \varphi_{r_\alpha R}) u_\alpha \right), \label{eq:Phi_1} \\
  \Delta\Phi_2 &= \begin{aligned}[t]
    & \nabla^\perp G_{\alpha,1} \cdot a\left( \varphi_\delta(1 - \varphi_{r_\alpha R}) u_\alpha \right) 
    + \nabla G_{\alpha,2} \nabla_{A_\alpha} \left( \varphi_\delta(1 - \varphi_{r_\alpha R}) u_\alpha \right) \label{eq:Phi_2} \\
     &\qquad - \varphi_\delta(1 - \varphi_{r_\alpha R}) \Phi_\alpha(A_\alpha, u_\alpha) 
    + \div\left( f_\alpha \left( \nabla \varphi_\delta (u_\alpha - \bar{u}_\alpha^1) - \nabla \varphi_{r_\alpha R} (u_\alpha - \bar{u}_\alpha^2) \right) \right). 
  \end{aligned}
\end{align}

To apply \cref{lem:Wente-Bethuel-Helein} to estimate \( \Phi_1 \), we modify \( G_{\alpha,1} \) so that it is supported in \( D \) while preserving its gradient on the neck domain \( T_{2\delta, r_\alpha R} \). Following the construction of \( \tilde u_\alpha \), define
\[
  \tilde G_{\alpha,1}(x) = \varphi_{4\delta}(x) \left[ (1 - \varphi_{r_\alpha R/2}(x))(G_{\alpha,1}(x) - \bar{G}_{\alpha,1}^2) + \bar{G}_{\alpha,1}^2 - \bar{G}_{\alpha,1}^1 \right],
\]
where
\[
  \bar{G}_{\alpha,1}^1 = \frac{1}{\lvert D_{8\delta} \setminus D_{4\delta} \rvert} \int_{D_{8\delta} \setminus D_{4\delta}} G_{\alpha,1}(x)\,dx, \quad
  \bar{G}_{\alpha,1}^2 = \frac{1}{\lvert D_{r_\alpha R} \setminus D_{r_\alpha R/2} \rvert} \int_{D_{r_\alpha R} \setminus D_{r_\alpha R/2}} G_{\alpha,1}(x)\,dx.
\]

Since \( \varphi_\delta(1 - \varphi_{r_\alpha R}) \) is supported in \( T_{2\delta, r_\alpha R} \) and \( \tilde G_{\alpha,1} = G_{\alpha,1} - \bar{G}_{\alpha,1}^1 \) on \( T_{2\delta, r_\alpha R} \), we can replace \( G_{\alpha,1} \) by \( \tilde G_{\alpha,1} \) in \eqref{eq:Laplace-D}, and instead consider the equation
\begin{equation}\label{eq:Phi_1-on-neck}
  \Delta \Phi_1 = \nabla^\perp \tilde G_{\alpha,1} \nabla \left( \varphi_\delta(1 - \varphi_{r_\alpha R}) u_\alpha \right),
\end{equation}
with vanishing boundary conditions on \( D \). Then, by \cref{lem:Wente-Bethuel-Helein}, we obtain
\begin{equation}\label{eq:nabla-Phi-1-intermediate}
  \lVert \nabla \Phi_1 \rVert_{L^2} \leq C \lVert \nabla \tilde G_{\alpha,1} \rVert_{L^{2,\infty}} \lVert \nabla \left( \varphi_\delta(1 - \varphi_{r_\alpha R}) u_\alpha \right) \rVert_{L^2}.
\end{equation}

We now estimate the \( L^{2,\infty} \) norm of \( \nabla \tilde G_{\alpha,1} \). Since \( L^2 \hookrightarrow L^{2,\infty} \) and \( \tilde G_{\alpha,1} = G_{\alpha,1} - \bar G_{\alpha,1} \) on \( T_{2\delta, r_\alpha R} \), it follows that
\begin{align}\label{eq:nabla-tilde-G-alpha-1-L2-infty-intermediate}
  \lVert \nabla \tilde G_{\alpha,1} \rVert_{L^{2,\infty}}
  &\leq C\left( \lVert \nabla \tilde G_{\alpha,1} \rVert_{L^2(\mathbb{R}^2\setminus B_{2\delta})}
    +\lVert \nabla  G_{\alpha,1} \rVert_{L^{2,\infty}(T_{2\delta,r_\alpha R})}
    +\lVert \nabla \tilde G_{\alpha,1} \rVert_{L^2(B_{r_\alpha R})}\right)\nonumber\\
  &\leq C\Bigl( 
    \begin{multlined}[t]
  \lVert \nabla \left( \varphi_{4\delta}(G_{\alpha,1}-\bar G_{\alpha,1}^1) \right) \rVert_{L^2(B_{8\delta}\setminus B_{2\delta})}
    +\lVert \nabla  G_{\alpha,1} \rVert_{L^{2,\infty}(T_{2\delta,r_\alpha R})}\\
    +\lVert \nabla \left( (1-\varphi_{r_\alpha R/2})(G_{\alpha,1}-\bar G_{\alpha,1}^2) \right) \rVert_{L^2(B_{r_\alpha R}\setminus B_{r_\alpha R/2})}
  \Bigr).
\end{multlined}
\end{align}
To estimate the three terms above, we first apply the Poincar\'e inequality:
\begin{align*}
  &\lVert \nabla \left(\varphi_{4\delta} (G_{\alpha,1}-\bar G^1_{\alpha,1})\right) \rVert_{L^2(B_{8\delta}\setminus B_{2\delta})}\\
  &\qquad\leq \lVert \nabla \varphi_{4\delta}(G_{\alpha,1}-\bar G_{\alpha,1}^1) \rVert_{L^2(B_{8\delta}\setminus B_{2\delta})}+\lVert \varphi_{4\delta}\nabla (G_{\alpha,1}-\bar G_{\alpha,1}^1) \rVert_{L^2(B_{8\delta}\setminus B_{2\delta})}\\
  &\qquad\leq \left( \frac{C}{\delta}\cdot C\delta +1 \right)\lVert \nabla (G_{\alpha,1}-\bar G_{\alpha,1}^1) \rVert_{L^2(B_{8\delta}\setminus B_{2\delta})}\leq C\lVert \nabla G_{\alpha,1} \rVert_{L^2(B_{8\delta}\setminus B_{2\delta})},
\end{align*}
Similarly,
\[
  \lVert \nabla \left( (1 - \varphi_{r_\alpha R/2})(G_{\alpha,1} - \bar G_{\alpha,1}^2) \right) \rVert_{L^2(B_{r_\alpha R} \setminus B_{r_\alpha R/2})}
  \leq C \lVert \nabla G_{\alpha,1} \rVert_{L^2(B_{r_\alpha R} \setminus B_{r_\alpha R/2})}.
\]
\begin{clm}
For all \( t \in [r_\alpha R/2, 4\delta] \), we have
\[
  \lVert \nabla G^{ij}_{\alpha,1} \rVert_{L^2(B_{2t} \setminus B_t)} \leq C\epsilon.
\]
\end{clm}
Indeed, for any \( t \in [r_\alpha R/2, 4\delta] \), from \eqref{eq:u-nabla-u-hodege-decompo}, we obtain
\begin{align*}
  \lVert \nabla G^{ij}_{\alpha,1} \rVert_{L^2(B_{2t} \setminus B_t)}
  &= \lVert \nabla^\perp G^{ij}_{\alpha,1} \rVert_{L^2(B_{2t} \setminus B_t)} \\
  &\leq \lVert f_\alpha \left( u_\alpha^j \nabla_{A_\alpha} u_\alpha^i - u_\alpha^i \nabla_{A_\alpha} u_\alpha^j \right) \rVert_{L^2(B_{2t} \setminus B_t)} + \lVert \nabla G_{\alpha,2}^{ij} \rVert_{L^2(B_{2t} \setminus B_t)} \\
  &\leq C \lVert f_\alpha \rVert_{C^0} \lVert u_\alpha \rVert_{C^0} \lVert \nabla_{A_\alpha} u_\alpha \rVert_{L^2(B_{2t} \setminus B_t)} + \lVert \nabla G_{\alpha,2}^{ij} \rVert_{L^2(B_{2t} \setminus B_t)} \\
  &\leq C\epsilon + \lVert \nabla G_{\alpha,2}^{ij} \rVert_{L^2(B_{2t} \setminus B_t)},
\end{align*}
where the last line uses the one-bubble assumption \eqref{eq:small-energy-on-each-annular-neck}, and \eqref{eq:f-alpha-C0} to control the \( C^0 \) norms.

Next, using \eqref{eq:G-alpha-2}, \cref{lem:div-elliptic-estimate}, and \cref{rmk:a_nabla_a_L_infinity}, we estimate
\begin{equation}\label{eq:nbla-G-alpha-2-L2}
  \lVert \nabla G^{ij}_{\alpha,2} \rVert_{L^2(B_{2t} \setminus B_t)} 
  \leq 2 \lVert \mathrm{tr}[F_{A_\alpha}, a_\alpha] \rVert_{L^2(B_{2t} \setminus B_t)} 
  \leq C \lVert a_\alpha \rVert_{L^\infty(B_{8\delta})} \lVert F_{A_\alpha} \rVert_{L^\infty(B_{8\delta})} \delta 
  \leq C\epsilon,
\end{equation}
where we may assume without loss of generality that \( \delta < \epsilon \). This completes the proof of the claim.
\bigskip

Hence, from \eqref{eq:nabla-tilde-G-alpha-1-L2-infty-intermediate}, we conclude:
\[
  \lVert \nabla \tilde G_{\alpha,1} \rVert_{L^{2,\infty}} 
  \leq C\epsilon + \lVert \nabla G_{\alpha,1} \rVert_{L^{2,\infty}(T_{2\delta, r_\alpha R})}.
\]
A similar argument gives:
\[
  \lVert \nabla G_{\alpha,1} \rVert_{L^{2,\infty}(T_{2\delta, r_\alpha R})}
  = \lVert \nabla^\perp G_{\alpha,1} \rVert_{L^{2,\infty}(T_{2\delta, r_\alpha R})} 
  \leq C \lVert \nabla_{A_\alpha} u_\alpha \rVert_{L^{2,\infty}(T_{2\delta, r_\alpha R})}
  + \lVert \nabla G_{\alpha,2} \rVert_{L^{2,\infty}(T_{2\delta, r_\alpha R})}.
\]
Using the estimate \eqref{eq:L-2-infty-u},
\[
  \lVert \nabla_{A_\alpha} u_\alpha \rVert_{L^{2,\infty}(T_{2\delta, r_\alpha R})} \leq C\epsilon.
\]
Likewise, by the proof of \eqref{eq:nbla-G-alpha-2-L2}, we have
\[
  \lVert \nabla G_{\alpha,2} \rVert_{L^{2,\infty}(T_{2\delta, r_\alpha R})} 
  \leq C \lVert \nabla G_{\alpha,2} \rVert_{L^{2}(T_{2\delta, r_\alpha R})} 
  \leq C\epsilon.
\]
Thus, we obtain the key estimate:
\[
  \lVert \nabla \tilde G_{\alpha,1} \rVert_{L^{2,\infty}} \leq C\epsilon.
\]
Therefore, from \eqref{eq:nabla-Phi-1-intermediate}, we deduce that
\begin{equation}\label{eq:nabla-Phi_1-L2}
    \lVert \nabla \Phi_1 \rVert_{L^2}
    \leq C\epsilon \lVert \nabla \left( \varphi_\delta(1 - \varphi_{r_\alpha R}) u_\alpha \right) \rVert_{L^2} 
    \leq C\epsilon \left( \lVert \nabla u_\alpha \rVert_{L^2}
      + \lVert \nabla \left( \varphi_\delta(1 - \varphi_{r_\alpha R}) \right) \rVert_{L^2} \right)
    \leq C\epsilon.
\end{equation}

To estimate \( \lVert \nabla \Phi_2 \rVert_{L^2} \), we recall from \cref{prop:small-energy-estimate} that, for \( x \in B_{2\delta} \setminus B_\delta \), one has
\begin{align*}
  \lvert u_\alpha - \bar u_\alpha^1 \rvert(x)
  &= \left\lvert \frac{1}{\lvert B_{2\delta} \setminus B_\delta \rvert} 
    \int_{B_{2\delta} \setminus B_\delta} \left( u_\alpha(x) - u_\alpha(y) \right) dy \right\rvert 
  \leq \frac{1}{\lvert B_{2\delta} \setminus B_\delta \rvert}
    \int_{B_{2\delta} \setminus B_\delta} \lvert u_\alpha(x) - u_\alpha(y) \rvert dy \\
  &\leq \left( \sup_{B_{2\delta} \setminus B_\delta} \lvert \nabla u_\alpha \rvert \right) 
    \frac{1}{\lvert B_{2\delta} \setminus B_\delta \rvert}
    \int_{B_{2\delta} \setminus B_\delta} \lvert x - y \rvert dy
  \leq C\epsilon,
\end{align*}
i.e.,
\begin{equation}\label{eq:mean-value-u1}
  \lvert u_\alpha - \bar u_\alpha^1 \rvert(x) \leq C\epsilon, \quad \forall x \in B_{2\delta} \setminus B_\delta.
\end{equation}
Similarly, for \( x \in B_{2r_\alpha R} \setminus B_{r_\alpha R} \), we have
\begin{equation}\label{eq:mean-value-u2}
  \lvert u_\alpha - \bar u_\alpha^2 \rvert(x) \leq C\epsilon.
\end{equation}

Now using the equation \eqref{eq:Phi_2} for \( \Phi_2 \), and applying the elliptic estimate and Sobolev embedding, we obtain
\begin{align*}
  \lVert \nabla \Phi_2 \rVert_{L^4}
  &\leq C \Bigl( \lVert \nabla G_{\alpha,2} \nabla_{A_\alpha}(\varphi_\delta(1 - \varphi_{r_\alpha R}) u_\alpha)
    - \varphi_\delta(1 - \varphi_{r_\alpha R}) \Phi_\alpha(A_\alpha, u_\alpha) \rVert_{L^2} \\
  &\qquad\qquad + \lVert \nabla^\perp G_{\alpha,1} \cdot a(\varphi_\delta(1 - \varphi_{r_\alpha R}) u_\alpha) \rVert_{L^2} \\
  &\qquad\qquad\qquad\qquad\qquad + \lVert f_\alpha(\nabla \varphi_\delta (u_\alpha - \bar u_\alpha^1) 
    - \nabla \varphi_{r_\alpha R}(u_\alpha - \bar u_\alpha^2)) \rVert_{L^2} \Bigr) \\
  &\leq C \Bigl( 
    \lVert \nabla G_{\alpha,2} \nabla_{A_\alpha}(\varphi_\delta(1 - \varphi_{r_\alpha R}) u_\alpha) \rVert_{L^2} 
    + \lVert \varphi_\delta(1 - \varphi_{r_\alpha R}) \Phi_\alpha(A_\alpha, u_\alpha) \rVert_{L^2} \\
  &\qquad\qquad + \lVert \nabla G_{\alpha,1} \rVert_{L^2(T_{2\delta, r_\alpha R})} \lVert a \rVert_{L^\infty} 
    + C\epsilon \lVert f_\alpha \rVert_{C^0} 
      \left( \lVert \nabla \varphi_\delta \rVert_{L^2} 
        + \lVert \nabla(1 - \varphi_{r_\alpha R}) \rVert_{L^2} \right) 
    \Bigr) \\
  &\leq C(\epsilon_0, \Lambda, \delta_1),
\end{align*}
where we use \eqref{eq:f-alpha-C0} and the following estimates: from the equation \eqref{eq:G-alpha-2} of \( G_{\alpha,2} \)  and \cref{rmk:a_nabla_a_L_infinity}, we have
\begin{align*}
  \lVert \nabla G_{\alpha,2} \rVert_{L^\infty(T_{2\delta, r_\alpha R})}
  &\leq C \lVert G_{\alpha,2} \rVert_{W^{2,2\alpha}(T_{2\delta, r_\alpha R})} 
  \leq C \lVert \div( \mathrm{tr}[F_A, a] ) \rVert_{L^{2\alpha}(T_{2\delta, r_\alpha R})} \\
  &\leq C \left( \lVert \nabla^2 a \rVert_{L^{2\alpha}(D)} \lVert a \rVert_{L^\infty(D)}
    + \lVert \nabla a \rVert_{L^{4\alpha}(D)}^2 
    + \lVert \nabla a \rVert_{L^{2\alpha}(D)} \lVert a \rVert_{L^\infty(D)}^2 \right) \\
  &\leq C(\epsilon_0, \Lambda, \delta_1).
\end{align*}
Also,
\begin{align*}
  \lVert \Phi_\alpha \rVert_{L^2(T_{2\delta, r_\alpha R})}
  &\leq C \Bigl(
    \lVert f_\alpha \rVert_{L^\infty(D)} \lVert a \cdot du \rVert_{L^2(T_{2\delta, r_\alpha R})} 
    + \lVert a \cdot a u \rVert_{L^2(T_{2\delta, r_\alpha R})} \\
  &\hspace*{12em} + \lVert \mu(u) \cdot \nabla \mu(u) \rVert_{L^2(T_{2\delta, r_\alpha R})}
  \Bigr) \\
  &\leq C \lVert a \cdot du \rVert_{L^2(T_{\delta, r_\alpha R})} + C\delta 
  \leq C(\epsilon_0, \Lambda, \delta_1).
\end{align*}
Moreover, for the term \( \lVert \nabla G_{\alpha,1} \rVert_{L^2(T_{2\delta, r_\alpha R})} \), we can proceed as before:
\begin{align*}
  \lVert \nabla G_{\alpha,1} \rVert_{L^2(T_{2\delta, r_\alpha R})}
  &= \lVert \nabla^\perp G_{\alpha,1} \rVert_{L^2(T_{2\delta, r_\alpha R})} \\
  &\leq \lVert f_\alpha \left( u_\alpha^j \nabla_{A_\alpha} u_\alpha^i - u_\alpha^i \nabla_{A_\alpha} u_\alpha^j \right) \rVert_{L^2(T_{2\delta, r_\alpha R})}
    + \lVert \nabla G_{\alpha,2}^{ij} \rVert_{L^2(T_{2\delta, r_\alpha R})} \\
  &\leq C \lVert f_\alpha \rVert_{C^0} \lVert u_\alpha \rVert_{C^0} 
    \lVert \nabla_{A_\alpha} u_\alpha \rVert_{L^2(T_{2\delta, r_\alpha R})}
    + \lVert \nabla G_{\alpha,2}^{ij} \rVert_{L^2(T_{2\delta, r_\alpha R})} \\
  &\leq C(\epsilon_0, \Lambda, \delta_1).
\end{align*}

By H\"older's inequality, we deduce:
\begin{equation}\label{eq:nabla-Phi_2-L2}
  \lVert \nabla \Phi_2 \rVert_{L^2} 
  = \lVert \nabla \Phi_2 \rVert_{L^2(D_{2\delta})} 
  \leq C \lVert \nabla \Phi_2 \rVert_{L^4(D_{2\delta})} \sqrt{\delta} 
  \leq C\epsilon.
\end{equation}
By \eqref{eq:Laplace-D}, we know that
\[
  \Delta D_\alpha = \div( f_\alpha \nabla_{A_\alpha} \tilde u_\alpha ) = \Delta \Phi_1 + \Delta \Phi_2,
\]
and since \( D_\alpha \) vanishes at infinity, we conclude that \( D_\alpha = \Phi_1 + \Phi_2 \). Therefore, from \eqref{eq:nabla-Phi_1-L2} and \eqref{eq:nabla-Phi_2-L2}, we conclude:
\[
  \lVert \nabla D_\alpha \rVert_{L^2} 
  \leq \lVert \nabla \Phi_1 \rVert_{L^2} + \lVert \nabla \Phi_2 \rVert_{L^2} 
  \leq C\epsilon.
\]

Finally, using the trivial estimate and \eqref{eq:f-nabla-u-by-nabla-D}, we obtain:
\begin{equation}\label{eq:EI-in-proof}
  \lVert \nabla_{A_\alpha} u_\alpha \rVert_{L^2(T_{\delta, r_\alpha R})}
  \leq \lVert \nabla_{A_\alpha} \tilde u_\alpha \rVert_{L^2(D)}
  \leq \lVert f_\alpha \nabla_{A_\alpha} \tilde u_\alpha \rVert_{L^2(D)}
  \leq C \lVert \nabla D_\alpha \rVert_{L^2(D)} 
  \leq C\epsilon.
\end{equation}

This completes the proof of \cref{lem:EI-E-A-u-D}.
\end{proof}

\step[Proof of the \( \alpha \)-energy ideneity]
We now complete the proof of the energy identity for the \( \alpha \)-energy of \( \alpha \)-YMH fields. By \cref{prop:equivalence-of-EI}, we only need to show \( \bar\mu=1 \) and \eqref{eq:equivalence-of-EI}. As a first step, we establish that \( \bar\mu = 1 \) via the following lemma.
\begin{lem}\label{lem:f_alpha-C0-equal-1}
  With the same notations and assumptions as above, we have
  \[
    \lim_{\alpha \to 1} \lVert f_\alpha \rVert_{C^0(D)} = \bar\mu = 1.
  \]
\end{lem}
\begin{proof}
  Integrating \eqref{eq:Pohozaev} with respect to \( t \) from \( r_\alpha R \) to \( \delta \), we obtain
  \[
    \int_{T_{\delta, r_\alpha R}} f_\alpha \left( \lvert \nabla_{A_\alpha;\partial_r} u_\alpha \rvert^2 
    - \frac{1}{2\alpha} \lvert \nabla_{A_\alpha} u_\alpha \rvert^2 \right) dv_g
    = \int_{r_\alpha R}^\delta \frac{\alpha - 1}{\alpha t} \int_{D_t} f_\alpha \lvert \nabla_{A_\alpha} u_\alpha \rvert^2 dv_g dt + \delta O(\delta).
  \]
  Clearly,
  \begin{multline*}
    \int_{r_\alpha R}^\delta \frac{\alpha - 1}{\alpha t} \int_{D_t} f_\alpha \lvert \nabla_{A_\alpha} u_\alpha \rvert^2 dv_g dt \\
    \geq \int_{r_\alpha R}^\delta \frac{\alpha - 1}{\alpha t} \int_{D_{r_\alpha R}} \lvert \nabla_{A_\alpha} u_\alpha \rvert^2 dv_g dt
    = \frac{2(\alpha - 1)}{\alpha} \ln \frac{\delta}{r_\alpha R} E(A_\alpha, u_\alpha; D_{r_\alpha R}),
  \end{multline*}
  where
  \[
    E(A_\alpha, u_\alpha; D_{r_\alpha R}) := \frac{1}{2} \int_{D_{r_\alpha R}} \lvert \nabla_{A_\alpha} u_\alpha \rvert^2.
  \]

  Recall \( E(\omega) = \frac{1}{2} \int_{\mathbb{R}^2} \lvert \nabla \omega \rvert^2 dv_g \) be the energy of a non-trivial harmonic sphere \( \omega \). Then, taking the limit \( \alpha \to 1 \), we have
  \[
    \lim_{\alpha \to 1} \int_{r_\alpha R}^\delta \frac{\alpha - 1}{\alpha t} \int_{D_t} f_\alpha \lvert \nabla_{A_\alpha} u_\alpha \rvert^2 dv_g dt
    \geq C \lim_{\alpha \to 1} \frac{\alpha - 1}{\alpha}  \ln \frac{\delta}{r_\alpha R}  E(\omega)
    = C \lim_{\alpha \to 1} \ln r_\alpha^{1 - \alpha} E(\omega).
  \]
  On the other hand, it is evident that
  \[
    \int_{T_{\delta, r_\alpha R}} f_\alpha \left( \lvert \nabla_{A_\alpha; \partial_r} u_\alpha \rvert^2 
    - \frac{1}{2\alpha} \lvert \nabla_{A_\alpha} u_\alpha \rvert^2 \right) dv_g
    \leq 2\Lambda' E(A_\alpha, u_\alpha; T_{\delta, r_\alpha R}) \leq C\epsilon^2,
  \]
  where we applied the energy identity estimate for \( (A_\alpha, u_\alpha) \), cf.~\eqref{eq:EI-in-proof}.
  Combining the above, we deduce
  \begin{align*}
    \lim_{\alpha \to 1} \ln r_\alpha^{1 - \alpha} E(\omega)
    &\leq C \lim_{\alpha \to 1} \int_{r_\alpha R}^\delta \frac{\alpha - 1}{\alpha t} \int_{D_t} f_\alpha \lvert \nabla_{A_\alpha} u_\alpha \rvert^2 dv_g dt \\
    &\leq C \left( \lim_{\alpha \to 1} \int_{T_{\delta, r_\alpha R}} f_\alpha \left( \lvert \nabla_{A_\alpha; \partial_r} u_\alpha \rvert^2 
      - \frac{1}{2\alpha} \lvert \nabla_{A_\alpha} u_\alpha \rvert^2 \right) dv_g 
      + \delta O(\delta) \right) \\
    &\leq C(\epsilon^2 + \delta^2).
  \end{align*}
  Therefore,
  \[
    \lim_{\alpha \to 1} \ln r_\alpha^{1 - \alpha} = 0 \quad \Rightarrow \quad \sqrt{\bar\mu} = \lim_{\alpha \to 1} r_\alpha^{1 - \alpha} = 1\implies \bar \mu=1.
  \]
  It follows that
  \[
    1 \leq \lVert f_\alpha \rVert_{C^0(D)}
    = \lVert (1 + \lvert \nabla_{A_\alpha} u_\alpha \rvert^2)^{\alpha - 1} \rVert_{C^0(D)}
    \leq \left( 1 + C r_\alpha^{-2} \right)^{\alpha - 1},
  \]
  where we use the bound \( \lVert \nabla_{A_\alpha} u_\alpha \rVert_{C^0(D)} \leq C / r_\alpha \), cf.~\eqref{eq:nabla-u-alpha-c0}. The proof is complete.
\end{proof}

Finally, note that,
\[
  \frac{1}{2} \int_{T_{\delta, r_\alpha R}} \left( 1 + \lvert \nabla_{A_\alpha} u_\alpha \rvert^2 \right)^{\alpha}
  = \frac{1}{2\alpha} \int_{T_{\delta, r_\alpha R}} f_\alpha \left( 1 + \lvert \nabla_{A_\alpha} u_\alpha \rvert^2 \right)
  \leq \frac{\lVert f_\alpha \rVert_{L^\infty(D)}}{2\alpha} \int_{T_{\delta, r_\alpha R}} \left( 1 + \lvert \nabla_{A_\alpha} u_\alpha \rvert^2 \right).
\]
By \eqref{eq:f-alpha-C0} and \eqref{eq:EI-in-proof}, the last term is bounded above by \( C(\delta^2 + \epsilon^2) \), provided that \( \alpha - 1 \) is sufficiently small. Taking the limit as \( \alpha \to 1 \) first, then letting \( R \to +\infty \), and finally \( \delta \to 0 \), we obtain \eqref{eq:equivalence-of-EI}. Together with \cref{lem:f_alpha-C0-equal-1}, which shows that \( \bar\mu = 1 \). Therefore, by \cref{prop:equivalence-of-EI}, the proof of the energy identity part of \cref{thm:main} is complete.

\section{The no neck property}\label{sec:NN}
In this section, we aim to prove that
\[
  \lim_{\delta \to 0} \lim_{R \to +\infty} \lim_{\alpha \to 1} \operatorname{osc}_{B_\delta \setminus B_{r_\alpha R}} u_\alpha 
  \mathpunct{:}= \lim_{\delta \to 0} \lim_{R \to +\infty} \lim_{\alpha \to 1} \sup_{x, y \in B_\delta \setminus B_{r_\alpha R}} \lvert u_\alpha(x) - u_\alpha(y) \rvert = 0,
\]
which will complete the proof of the no-neck part of \cref{thm:main}. The argument follows a similar structure to the energy identity established in \cref{sec:EI}, with two essential ingredients. First, we rely on the uniform bound \( \lVert f_\alpha \rVert_{C^0(D)} \to 1 \) as \( \alpha \to 1 \), shown in \cref{lem:f_alpha-C0-equal-1}. Second, we refine the key estimates by upgrading from the weak norm \( L^{2,\infty} \) to the strong Lorentz norm \( L^{2,1} \). This refinement is enabled by the improved elliptic estimate in \cref{lem:div-elliptic-estimate}, together with the localized energy identity established in \cref{lem:EI-E-A-u-D}.

Recall from the construction of \( \tilde u_\alpha \) that \( \tilde u_\alpha = u_\alpha-\bar u_\alpha^1 \) on \( B_\delta \setminus B_{2r_\alpha R} \) and \( \tilde u_\alpha \) vanishes at infinity. Hence,
\begin{equation}\label{eq:NN-nabla-tilde-u-L21}
  \sup_{x, y \in B_\delta \setminus B_{2r_\alpha R}} \lvert u_\alpha(x) - u_\alpha(y) \rvert
  \leq \sup_{x, y \in \mathbb{R}^2} \lvert \tilde u_\alpha(x) - \tilde u_\alpha(y) \rvert
  \leq 2\lVert \tilde u_\alpha \rVert_{C^0}
  \leq C \lVert \nabla \tilde u_\alpha \rVert_{L^{2,1}},
\end{equation}
where we used Lorentz space estimates for functions vanishing at infinity (see \cref{lem:embedding-L21-C0}) in the last inequality.

To estimate \( \lVert \nabla \tilde u_\alpha \rVert_{L^{2,1}} \), we observe from \eqref{eq:Hodge-decomposition} that
\[
  \Delta Q_\alpha = -\div \left( (f_\alpha - 1) \nabla_{A_\alpha}^\perp \tilde u_\alpha \right).
\]
Since \( \tilde u_\alpha \) is supported in \( B_{2\delta} \) and \( Q_\alpha \) vanishes at infinity, \cref{lem:div-elliptic-estimate} implies
\[
  \lVert \nabla Q_\alpha \rVert_{L^{2,1}} \leq C \lVert (f_\alpha - 1) \nabla_{A_\alpha}^\perp \tilde u_\alpha \rVert_{L^{2,1}}
  = C \left\lVert \left(1 - \frac{1}{f_\alpha} \right) f_\alpha \nabla_{A_\alpha} \tilde u_\alpha \right\rVert_{L^{2,1}}.
\]
Thus, by \cref{lem:f_alpha-C0-equal-1}, for sufficiently small \( \alpha - 1 > 0 \),
\[
  \lVert \nabla Q_\alpha \rVert_{L^{2,1}} \leq \frac{1}{2} \lVert f_\alpha \nabla_{A_\alpha} \tilde u_\alpha \rVert_{L^{2,1}},
\]
which, together with \eqref{eq:Hodge-decomposition}, yields
\begin{equation}\label{eq:D-L21}
  \lVert f_\alpha \nabla_{A_\alpha} \tilde u_\alpha \rVert_{L^{2,1}} \leq 2 \lVert \nabla D_\alpha \rVert_{L^{2,1}}.
\end{equation}

We now proceed similarly to the proof of the energy identity (see \cref{sec:EI}). Applying the conservation law \eqref{eq:conservation-law-component}, we recall \eqref{eq:Laplace-D}:
\begin{align*}
  \operatorname{div}(f_\alpha \nabla_{A_\alpha} \tilde u_\alpha)
  &= \nabla^\perp G_{\alpha,1} \cdot \nabla_{A_\alpha}(\varphi_\delta(1 - \varphi_{r_\alpha R}) u_\alpha)
    + \nabla G_{\alpha,2} \cdot \nabla_{A_\alpha}(\varphi_\delta(1 - \varphi_{r_\alpha R}) u_\alpha) \\
  &\qquad - \varphi_\delta(1 - \varphi_{r_\alpha R}) \Phi_\alpha(A_\alpha, u_\alpha) 
   + \operatorname{div} \left( f_\alpha \left[ \nabla \varphi_\delta (u_\alpha - \bar{u}_\alpha^1)
    - \nabla \varphi_{r_\alpha R} (u_\alpha - \bar{u}_\alpha^2) \right] \right).
\end{align*}
Let \( \Phi_1 \) and \( \Phi_2 \) be the solutions to \eqref{eq:Phi_1} and \eqref{eq:Phi_2} with vanishing boundary conditions on \( D \), respectively. From \eqref{eq:Phi_1-on-neck} and \cref{lem:Wente-L21}, here the Jacobian structure in \eqref{eq:Phi_1-on-neck} once again plays  a crucial role, we obtain
\begin{equation}\label{eq:nabla-Phi_1-L21}
  \lVert \nabla \Phi_1 \rVert_{L^{2,1}} \leq C \lVert \nabla \tilde G_{\alpha,1} \rVert_{L^2}
  \lVert \nabla(\varphi_\delta(1 - \varphi_{r_\alpha R}) u_\alpha) \rVert_{L^2}.
\end{equation}
Using a Poincar\'e inequality and arguments as in the energy identity (see \eqref{eq:nabla-tilde-G-alpha-1-L2-infty-intermediate}), we estimate:
\[
  \lVert \nabla \tilde G_{\alpha,1} \rVert_{L^2}
  \leq C \lVert \nabla G_{\alpha,1} \rVert_{L^2(T_{8\delta,r_\alpha R/2})}
  \leq C \left( \lVert f_\alpha \nabla_{A_\alpha} u_\alpha \rVert_{L^2(T_{8\delta,r_\alpha R/2})} + \lVert \nabla G_{\alpha,2} \rVert_{L^2(T_{8\delta,r_\alpha R/2})} \right).
\]
From \eqref{eq:G-alpha-2}, the elliptic estimate yields:
\[
  \lVert \nabla G_{\alpha,2} \rVert_{L^4(T_{8\delta,r_\alpha R/2})}
  \leq C \lVert G_{\alpha,2} \rVert_{W^{2,2}(B_{8\delta})}
  \leq C \lVert a \rVert_{L^\infty} \lVert F_{A_\alpha} \rVert_{L^2} \leq C(\delta_1),
\]
and thus by H\"older's inequality,
\[
  \lVert \nabla G_{\alpha,2} \rVert_{L^2(T_{8\delta,r_\alpha R/2})}
  \leq C(\delta_1) \delta^{1/2} \leq C \epsilon.
\]
Combining with \eqref{eq:f-alpha-C0} and energy identity \eqref{eq:EI-in-proof},
\[
  \lVert \nabla \tilde G_{\alpha,1} \rVert_{L^2} 
  \leq C \lVert \nabla G_{\alpha,1} \rVert_{L^2(T_{8\delta,r_\alpha R/2})}
  \leq C \left( \lVert f_\alpha \rVert_{C^0} \lVert \nabla_{A_\alpha} u_\alpha \rVert_{L^2(T_{8\delta,r_\alpha R/2})} + C\epsilon \right) \leq C \epsilon.
\]
Plugging this into \eqref{eq:nabla-Phi_1-L21} gives
\begin{equation}\label{eq:L21-Phi_1}
  \lVert \nabla \Phi_1 \rVert_{L^{2,1}} \leq C \epsilon \left( \lVert \nabla u_\alpha \rVert_{L^2}
  + \lVert \nabla (\varphi_\delta(1 - \varphi_{r_\alpha R})) \rVert_{L^2} \right) \leq C \epsilon.
\end{equation}

For \( \lVert \nabla \Phi_2 \rVert_{L^{2,1}} \), by \eqref{eq:Phi_2} and Sobolev embedding \( W^{2,4/3} \hookrightarrow W^{1,4} \), \cref{lem:div-elliptic-estimate}, and \cref{lem:elliptic-estimate}, we obtain:
\begin{align*}
  \lVert \nabla \Phi_2 \rVert_{L^{2,1}} &\leq C \Bigl(
    \lVert \nabla G_{\alpha,1} \rVert_{L^{4/3}(T_{2\delta,r_\alpha R/2})} \lVert a \rVert_{L^\infty}
    + \lVert \nabla G_{\alpha,2} \cdot \nabla_{A_\alpha} (\varphi_\delta(1 - \varphi_{r_\alpha R}) u_\alpha) \rVert_{L^{4/3}(T_{2\delta,r_\alpha R/2})} \\
    &\qquad + \lVert \Phi_\alpha(A_\alpha, u_\alpha) \rVert_{L^{4/3}(T_{2\delta,r_\alpha R/2})}
    + \lVert f_\alpha \left( \nabla \varphi_\delta(u_\alpha - \bar u_\alpha^1)
    - \nabla \varphi_{r_\alpha R}(u_\alpha - \bar u_\alpha^2) \right) \rVert_{L^{2,1}} \Bigr).
\end{align*}

Each term on the right-hand side is bounded by \( C(\sqrt{\delta} + \epsilon) \), using estimates analogous to those in the proof of energy identity. In fact, clearly
\[
  \lVert \nabla G_{\alpha,1} \rVert_{L^{4/3}(T_{2\delta,r_\alpha R/2})}
  \leq C\lVert \nabla G_{\alpha,1} \rVert_{L^{2}(T_{8\delta,r_\alpha R/2})}
  \leq C\epsilon,
\]
Next, by the energy identity \eqref{eq:EI-in-proof},
\[
  \lVert \nabla G_{\alpha,2}  \nabla_{A_\alpha} \left( \varphi_\delta (1 - \varphi_{r_\alpha R}) u_\alpha \right) \rVert_{L^{4/3}}
  \leq C \lVert \nabla G_{\alpha,2} \rVert_{L^4(T_{2\delta,r_\alpha R/2})}
  \lVert 1 + \lvert \nabla_{A_\alpha} u_\alpha \rvert \rVert_{L^2(T_{2\delta,r_\alpha R/2})}
  \leq C (\delta + \epsilon),
\]
and
\[
  \lVert \Phi_\alpha(A_\alpha, u_\alpha) \rVert_{L^{4/3}(T_{2\delta,r_\alpha R/2})}
  \leq C \left( \lVert a \rVert_{L^\infty}\lVert f_\alpha \rVert_{C^0} \lVert \nabla_{A_\alpha} u_\alpha \rVert_{L^{4/3}(D_{2\delta})}
  + \lVert \mu(u_\alpha) \rVert_{L^{4/3}(D_{2\delta})} \right)
  \leq C \sqrt{\delta}.
\]
Finally, we estimate the contribution from the cut-off derivatives:
\begin{align*}
  &\lVert f_\alpha \left( \nabla \varphi_\delta (u_\alpha - \bar u_\alpha^1)
  - \nabla \varphi_{r_\alpha R} (u_\alpha - \bar u_\alpha^2) \right) \rVert_{L^{2,1}} \\
  &\qquad \leq C \left( \lVert (u_\alpha - \bar u_\alpha^1) \nabla \varphi_\delta \rVert_{L^{2,1}}
  + \lVert (u_\alpha - \bar u_\alpha^2) \nabla \varphi_{r_\alpha R} \rVert_{L^{2,1}} \right)
  \leq C \epsilon,
\end{align*}
where we have used \eqref{eq:mean-value-u1} and \eqref{eq:mean-value-u2}, together with the uniform bound \eqref{eq:f-alpha-C0} on \( f_\alpha \) and support properties of \( \varphi_t \).

Hence,
\begin{equation}\label{eq:L21-Phi_2}
  \lVert \nabla \Phi_2 \rVert_{L^{2,1}} \leq C \left( \sqrt{\delta} + \epsilon \right) \leq C \epsilon.
\end{equation}
Since \( D_\alpha = \Phi_1 + \Phi_2 \), combining \eqref{eq:L21-Phi_1} and \eqref{eq:L21-Phi_2} gives
\[
  \lVert \nabla D_\alpha \rVert_{L^{2,1}} \leq \lVert \nabla \Phi_1 \rVert_{L^{2,1}} + \lVert \nabla \Phi_2 \rVert_{L^{2,1}} \leq C \epsilon.
\]

Finally, by \eqref{eq:D-L21}, we conclude:
\[
  \lVert \nabla \tilde u_\alpha \rVert_{L^{2,1}} \leq \lVert f_\alpha \nabla_{A_\alpha} \tilde u_\alpha \rVert_{L^{2,1}}+C\lVert f_\alpha \rVert_{C^0}\lVert a_\alpha \rVert_{L^\infty(D_{2\delta})}\sqrt{\delta}
  \leq 2 \lVert \nabla D_\alpha \rVert_{L^{2,1}}+C\sqrt{\delta} \leq C \epsilon.
\]
Combining this with \eqref{eq:NN-nabla-tilde-u-L21} completes the proof of the no-neck property stated in \cref{thm:main}.

\bigskip

\bigskip
\end{document}